\documentclass[12pt]{amsart}
\usepackage{amsmath,amsthm,amssymb,amscd,epsfig,latexsym,graphicx,eucal,layout,fancyhdr,tikz,bm,mathrsfs}
\usepackage[cmtip,all]{xy}

\tikzset{node distance=2cm, auto}

\def\Empty{}
\newcommand\oplabel[1]{
  \def\OpArg{#1} \ifx \OpArg\Empty {} \else
    \label{#1}
  \fi}

\long\def\realfig#1#2#3#4{
\begin{figure}[t]
\centerline{\psfig{figure=#2,width=#4}}
\caption[#1]{#3}
\oplabel{#1}
\end{figure}}

\newcommand{\comm}[1]{}

\input{psbox}

\newtheorem{theorem}{Theorem}[section]
\newtheorem{lemma}[theorem]{Lemma}

\newtheorem{corollary}[theorem]{Corollary}

\newtheoremstyle{named}{}{}{\itshape}{}{\bfseries}{.}{.5em}{\thmnote{#3 }#1}
\theoremstyle{named}
\newtheorem*{namedtheorem}{Theorem}

\theoremstyle{definition}
\newtheorem{definition}[theorem]{Definition}
\newtheorem{example}[theorem]{Example}

\newcommand{\Real}{\operatorname{Re}}
\newcommand{\rad}{\operatorname{rad}}

\newcommand{\bd}{\partial}

\newcommand{\es}{\emptyset}
\newcommand{\ov}{\overline}
\newcommand{\sm}{\smallsetminus}
\newcommand{\ve}{\varepsilon}

\newcommand{\vs}{\vspace{2mm}}

\newcommand{\CC}{{\mathbb C}}

\newcommand{\RR}{{\mathbb R}}

\newcommand{\TT}{{\mathbb T}}
\newcommand{\ZZ}{{\mathbb Z}}

\newcommand{\DD}{{\mathbb D}}

\newcommand{\Chat}{\hat{\mathbb C}}

\newcommand{\AT}{\tilde{A}}
\newcommand{\IT}{\tilde{I}}
\newcommand{\UT}{\tilde{U}}
\newcommand{\XT}{\tilde{X}}
\newcommand{\gT}{\tilde{g}}
\newcommand{\sT}{\tilde{s}}
\newcommand{\aT}{\tilde{a}}
\newcommand{\bT}{\tilde{b}}
\newcommand{\cT}{\tilde{c}}

\newcommand{\pT}{\tilde{p}}
\newcommand{\VT}{\tilde{V}}

\newcommand{\rt}{R_{\theta}}
\newcommand{\rot}{\rho_{\theta}}
\newcommand{\rott}{\rho_{\theta,t}}
\newcommand{\ET}{\tilde{\mathscr E}}
\newcommand{\be}{{\bf e}}

\newcommand{\bfit}{\it \bfseries}

\psfordvips

\newcommand{\thmref}[1]{Theorem~\ref{#1}}

\newcommand{\lemref}[1]{Lemma~\ref{#1}}
\newcommand{\defref}[1]{Definition~\ref{#1}}
\newcommand{\corref}[1]{Corollary~\ref{#1}}
\newcommand{\figref}[1]{Fig.~\ref{#1}}
\newcommand{\exref}[1]{Example~\ref{#1}}

\begin{document}

\title[Uniformization of holomorphically moving disks]{Conformal fitness and uniformization of holomorphically moving disks}
\author[S. Zakeri]{Saeed Zakeri}
\address{S. Zakeri, Department of Mathematics, Queens College and Graduate Center of CUNY, New York}
\email{saeed.zakeri@qc.cuny.edu}

\subjclass[2010]{37Fxx, 30C85}

\date{February 11, 2013}

\begin{abstract}
Let $\{ U_t \}_{t \in \DD}$ be a family of topological disks on the Riemann sphere containing the origin $0$ whose boundaries undergo a holomorphic motion over the unit disk $\DD$. We study the question of when there exists a family of Riemann maps $g_t:(\DD,0) \to (U_t,0)$ which depends holomorphically on the parameter $t$. We give five equivalent conditions which provide analytic, dynamical and measure-theoretic characterizations for the existence of the family $\{ g_t \}_{t \in \DD}$, and explore the consequences.  
\end{abstract}

\maketitle

\section{Introduction} \label{sec:intro}

This paper will address a problem in 2-dimensional conformal geometry which has its origin in
holomorphic dynamics. It is motivated by the following observation of Dennis Sullivan in \cite{Su}: Let $\{ f_t \}$ be a family of rational maps of the Riemann sphere depending holomorphically on the complex parameter $t$ in the unit disk $\DD$. Suppose
each $f_t$ has a Siegel disk $U_t$ centered at $0$ whose topological boundary
$\bd U_t$ undergoes a holomorphic motion over $\DD$. Then there is a family of Riemann maps $g_t: (\DD,0) \to (U_t,0)$ which depends
holomorphically on $t$. In particular, if $\rad(U_t,0) = |g'_t(0)|$ denotes the conformal radius of the pointed disk $(U_t,0)$, then $t \mapsto \log \rad (U_t,0)$ is harmonic in $\DD$. Sullivan's result has been used in several recent
works on Siegel disks (see e.g. \cite{BC}, \cite{BP}, \cite{Z2} and \cite{Z3}).
The nature of the function $t \mapsto \log \rad (U_t,0)$ under various assumptions on
the family $\{ f_t \}$ of holomorphic maps has been studied in \cite{BC}. \vs

By a {\bfit disk} we mean a simply connected domain in the Riemann sphere
$\Chat$ whose complement has more than one point. Let $\{ (U_t,c_t) \}_{t \in \DD}$
be a family of pointed disks, where the marked center $c_t \in U_t$ depends holomorphically on $t$. Suppose the boundaries of these disks undergo a holomorphic motion $\varphi_t: \bd U_0 \to \bd U_t$ over $\DD$. This means $\{ \varphi_t \}_{t \in \DD}$ is a family of injections depending holomorphically on $t$, with $\varphi_0=\text{id}$. We are interested in the question of whether there exists a family $g_t : (\DD,0) \to (U_t,c_t)$ of Riemann maps which depends holomorphically on $t$. It is not hard to see that such a family does not always exist. For example, consider the real affine map $\varphi_t: z \mapsto z + t \, \bar{z}$ for
$|t|<1$, and set $U_0=\DD$ and $U_t=\varphi_t(U_0)$. Evidently $\varphi_t : \bd U_0
\to \bd U_t$ defines a holomorphic motion over $\DD$. By the Koebe $1/4$-theorem,
the conformal radius $\rad (U_t,0)$ of the pointed ellipse $(U_t,0)$ is at most $4$
times its inner radius of $1-|t|$, hence the average value of $\log \rad (U_t,0)$ over the circle $|t|=1-\ve$ tends to $-\infty$ as $\ve \to 0$. On the other hand, if there were a family of Riemann maps $g_t: (\DD,0) \to (U_t,0)$ depending
holomorphically on $t$, the function $t \mapsto \log \rad (U_t,0) = \log |g_t'(0)|$
would be harmonic in $\DD$, taking the value $0$ at $t=0$. Hence the average
value of $\log \rad (U_t,0)$ over the circle $|t|=1-\ve$ would remain $0$ irrespective of what $\ve$ is. Note that in this example the boundary of the complementary disk $V_t=\Chat \sm \ov{U}_t$ undergoes the holomorphic motion $(z,t) \mapsto z+t/z$, which already provides a holomorphically varying family of Riemann maps $(\Chat \sm \ov{\DD} , \infty) \to (V_t, \infty)$. Hence $\log \rad(V_t,\infty)$ must be a harmonic function of $t$ (it is actually the constant function $0$). \vs

To understand the obstructions involved, it will be useful to first consider the problem in the ``static'' case where we merely have a pair of pointed disks $(U,c)$ and $(\UT,\cT)$ in $\Chat$ and a homeomorphism $\varphi: \bd U \to \bd \UT$ which is compatible with the embedding in the sphere. This sets the stage for the ``dynamic'' case where the disk boundaries move holomorphically, but we believe it also deserves to be investigated on its own merits. Our formulation and solution of the problem turns out to depend on the following key notion: We say that $\varphi$ is {\bfit conformally fit} if there are Riemann maps $g: (\DD,0) \to (U,c)$ and $\gT: (\DD,0) \to (\UT,\cT)$ for which the relation
$$
\gT = \varphi \circ g
$$
holds almost everywhere on the unit circle $\TT = \{ z \in \CC : |z|=1 \}$ (\defref{CA}). This is
equivalent to the relation $[\gT] = [\varphi] \circ [g]$ between the induced homeomorphisms defined on the corresponding spaces of prime ends (\lemref{EQ}). It is easy to show that $\varphi: \bd U \to \bd \UT$ is conformally fit if and only if it is the radial limit map of a biholomorphism $(U,c) \to (\UT,\cT)$ (\thmref{adan}). \vs

Conformal fitness can also be characterized in dynamical terms using
rotations. For each $\theta \in \RR/\ZZ$, let $\rt$ denote the rigid rotation $z \mapsto e^{2 \pi i \theta} z$ on $\DD$. Given a disk $U \subset \Chat$ centered at $c$ and a Riemann map $g: (\DD,0) \to (U,c)$, the map $\rot = g \circ \rt \circ g^{-1} : (U,c) \to (U,c)$ is called the {\bfit intrinsic rotation} of $U$ by the angle $\theta$ about the center $c$. The word ``intrinsic'' refers to the fact that $\rot$ depends only on the pointed disk $(U,c)$ and not on the choice of the conformal isomorphism $g$: By the Schwarz lemma, $g$ is unique up to pre-composition with rigid rotations, which trivially commute with $\rt$. In \thmref{adyn} we show that $\varphi: \bd U \to \bd \UT$ is conformally fit if and only if it is the radial limit map of a homeomorphism $\Phi: (U,c) \to (\UT,\cT)$ such that $\Phi \circ \rot = \tilde{\rho}_{\theta} \circ \Phi$ for some (equivalently, every) irrational $\theta$. \vs

In $\S$\ref{sec:meas} we find another characterization of conformal fitness
based on a refinement of the notion of harmonic measure. According to Fatou, every Riemann map $g: (\DD,0) \to (U,c)$ has the radial limit $g(a) = \lim_{r \to 1} g(ra) \in \Chat$ for almost every choice of $a \in \TT$ with respect to Lebesgue measure.
We use the notation $g^{-1}(X)= \{ a \in \TT : g(a) \ \text{exists and is in} \ X \}$
whenever $X$ is a subset of $\bd U$. The almost everywhere defined radial limit
map $g: \TT \to \bd U$ pushes the normalized Lebesgue measure $m$ on $\TT$ forward
to the {\bfit harmonic measure} $\omega = g_{\ast} m$ as seen from $c$,
which is a Borel probability measure supported on $\bd U$. Explicitly, $\omega(X)=m(g^{-1}(X))$ for every Borel set $X \subset \bd U$. It is not hard to show that when $U,\UT$ are Jordan domains with marked centers $c, \cT$, a homeomorphism $\varphi: \bd U \to
\bd \UT$ is conformally fit if and only if $\varphi_{\ast} \omega =
\tilde{\omega}$ (\thmref{HM1}). For arbitrary disks, however, respecting the harmonic measure does not guarantee conformal fitness (\exref{dom}). To get around this problem, we refine the notion of harmonic measure to account for different ways of accessing a boundary point from within the disk. Up to a set of harmonic measure zero, the boundary of a disk $U$ decomposes into the disjoint union
\begin{equation}\label{yek}
\bd U = \bd U^1 \cup \bd U^2,
\end{equation}
where $\bd U^1$ is the set of {\bfit uniaccessible} points at which exactly one {\bfit ray} (= hyperbolic geodesic starting at $c$) lands, while $\bd U^2$ is the set of {\bfit biaccessible} points where precisely two rays land. Once a boundary point $s \in \bd U^1$ is marked, there is a simple way of distinguishing the two rays that land at each biaccessible point, which we designate by $-$ and $+$ (see $\S$\ref{sec:meas} for details). Taking the preimages of each side of \eqref{yek} under the unique Riemann map $g: (\DD,0) \to (U,c)$ with $g(1)=s$, we obtain a
corresponding measurable decomposition
$$
\TT = A^1 \cup A^{2-} \cup A^{2+}
$$
up to a set of Lebesgue measure zero. This gives a decomposition
\begin{equation}\label{omdec}
\omega = \alpha + \beta^- + \beta^+
\end{equation}
of the harmonic measure, where
$$
\alpha(X) = m(g^{-1}(X) \cap A^1) \qquad \text{and} \qquad
\beta^{\pm}(X) = m(g^{-1}(X) \cap A^{2\pm})
$$
for every Borel set $X \subset \bd U$. Note that $\alpha$ is intrinsic, but the measures $\beta^{\pm}$ depend on the choice of the marked boundary point $s$. This
allows a measure-theoretic characterization of conformal fitness: A homeomorphism $\varphi: \bd U \to
\bd \UT$ with $\varphi(s)=\tilde{s}$ is conformally fit if and only if $\varphi_{\ast} \alpha =
\tilde{\alpha}$ and $\varphi_{\ast} \beta^{\pm} = \tilde{\beta}^{\pm}$
(\thmref{mtadapt}). \vs

Back to our original question, suppose now that $\{ (U_t, c_t) \}_{t \in \DD}$ is a
family of pointed disks in $\Chat$, where the marked center $c_t$ depends holomorphically on $t$. Assume that the boundaries of these disks undergo a holomorphic motion $\varphi_t: \bd U_0 \to \bd U_t$ over $\DD$. In $\S$\ref{sec:hm} we show that the existence of a holomorphically varying family
of Riemann maps $g_t:(\DD,0) \to (U_t,c_t)$ is equivalent to conformal
fitness of the map $\varphi_t: \bd U_0 \to \bd U_t$ for each parameter $t$.
The above results then immediately translate into analytic, dynamical and
measure-theoretic characterizations in the holomorphically moving case.
But, as it is well known, the existence of a holomorphic motion regulates the disk boundaries in a rather strong form; for example, it implies that the boundaries  are all quasiconformally equivalent. So it should come as no surprise that there are alternative characterizations of conformal fitness in the holomorphically moving case that are sharper than the static case. The following is proved in $\S$\ref{sec:hm}: \vs

\begin{namedtheorem}[Main]
Let $\{ (U_t, c_t) \}_{t \in \DD}$ be a family of pointed disks in $\Chat$, where the center $c_t$ depends holomorphically on $t$. Suppose the boundaries of these disks undergo a holomorphic motion $\varphi_t: \bd U_0 \to \bd U_t$ over $\DD$. Then the following conditions are equivalent: \vs
\begin{enumerate}
\item[(i)]
There is a family of Riemann maps $g_t : (\DD,0) \to (U_t,c_t)$ which depends
holomorphically on $t$. \vs
\item[(ii)]
$\varphi_t: \bd U_0 \to \bd U_t$ is a ``trivial motion'' in the sense that it extends to a holomorphic motion $\Phi_t: (\Chat, c_0) \to (\Chat,c_t)$ such that the restriction $\Phi_t: (U_0,c_0) \to (U_t,c_t)$ is a biholomorphism for each $t$. \vs
\item[(iii)]
$\varphi_t : \bd U_0 \to \bd U_t$ is conformally fit for each $t$. \vs
\item[(iv)]
There is an irrational $\theta \in \RR/\ZZ$ for which the intrinsic rotation $\rott: (U_t,c_t) \to
(U_t,c_t)$ depends holomorphically on $t$. \vs
\item[(v)]
$\varphi_t: \bd U_0 \to \bd U_t$ respects the harmonic measure as seen from the center: $(\varphi_t)_{\ast} \omega_0 = \omega_t$ for each $t$. \vs
\item[(vi)]
The map $t \mapsto \log \rad(U_t,c_t)$ is harmonic in $\DD$. 
\end{enumerate}
\end{namedtheorem}

The equivalence of the first four conditions depends on elementary properties of holomorphic motions, with (iv) $\Longrightarrow$ (i) essentially being Sullivan's argument. But the fact that the last two are also equivalent to them is rather curious since initially (v) and (vi) appear to be weaker conditions. Observe that the decomposition \eqref{omdec}, which was essential in our measure-theoretic characterization of fitness in the static case, plays no role in the holomorphically moving case. The equivalence (v) $\Longleftrightarrow$ (vi) follows from the potential-theoretic characterization of harmonic measure as the equilibrium measure minimizing the energy integral, while the implication (v) $\Longrightarrow$ (i) follows from this by taking suitable double-covers of disks under which some biaccessible points are rendered uniaccessible, a construction which is explained in \S \ref{sec:zdc}. \vs

Here are two corollaries that may be of independent interest:

\begin{corollary}
The only holomorphic motions of a pointed disk which keep the conformal radius fixed are the trivial motions.
\end{corollary}

This follows from the implication (vi) $\Longrightarrow$ (ii) in the Main Theorem. Note that when $t \mapsto \log \rad(U_t,c_t)$ is harmonic, there is a suitable rescaling of the motion which makes the conformal radius constant in $t$ (see \S \ref{sec:hm}). This shows that the above corollary is in fact equivalent to (vi) $\Longrightarrow$ (ii).

\begin{corollary}
Take a family $\{ U_t \}_{t \in \DD}$ of disks in $\Chat$ whose boundaries $\{ \bd U_t \}_{t \in \DD}$ undergo a holomorphic motion over $\DD$. Let $h_t:\DD \to U_t$ be any Riemann map with the power series expansion
$$
h_t(z)= a_0(t) + a_1(t) \, z + a_2(t) \, z^2 + \cdots
$$
If the map $t \mapsto a_n(t)$ is holomorphic for $n=0,1$, then it is holomorphic for all $n$.
\end{corollary}

In fact, the center $c_t=a_0(t)$ is holomorphic in $t$ and the logarithm of the conformal radius $\log \rad(U_t,c_t) = \log |a_1(t)|$ is harmonic in $t$, so by the Main Theorem there is a family $g_t:(\DD,0) \to (U_t,c_t)$ of Riemann maps which depends holomorphically on $t$. By the Schwarz lemma, $h_t(z)=g_t(\lambda_t z)$ for some $\lambda_t \in \CC$ with $|\lambda_t|=1$. As $\lambda_t=a_1(t)/g'_t(0)$ is holomorphic in $t$, it must be constant. This proves the holomorphic dependence of $h_t$ and therefore all the coefficients $a_n(t)=h_t^{(n)}(0)/n!$ on $t$. \vs

Let us revisit the dynamical setting in the beginning of this introduction, where $\{ f_t : \Chat \to \Chat \}_{t \in \DD}$ is a holomorphic family of rational maps with fixed Siegel disks $U_t$ of a given rotation number $\theta$ centered at the origin. Suppose the Julia set of $f_t$ moves holomorphically over $\DD$. This motion restricts to a holomorphic motion of $\bd U_t$, and the condition (iv) above holds automatically since the intrinsic rotation $\rott: (U_t,0) \to (U_t,0)$ is just $f_t$. Hence the Main Theorem implies the existence of a holomorphically varying family of Riemann maps $g_t: (\DD,0) \to (U_t,0)$ and $t \mapsto \log \rad(U_t,0)$ is harmonic, which is Sullivan's result in \cite{Su}. It also follows from the condition (v) that any motion of a Siegel disk must respect the harmonic measure on the boundary as seen from the fixed point inside. \vs

The Main Theorem reveals the special nature of Siegel disks, as the analogous statement for other simply connected Fatou components is generally false. For example, the Julia set $J_t$ of the quadratic polynomial $f_t: z \mapsto tz+z^2$ undergoes a holomorphic motion $\varphi_t: J_0=\TT \to J_t$ over the unit disk $|t|<1$. For each $t \in \DD$, the Jordan curve $J_t$ separates the sphere into two simply connected Fatou components $U_t$ and $V_t$, where $U_t$ is the basin of attraction of $0$ and $V_t$ is that of $\infty$ under the iterations of $f_t$. It is well known that $\varphi_t$ seen as a motion $\bd V_0 \to \bd V_t$ is trivial (for example, $t \mapsto \rad(V_t, \infty)$ is identically $1$). If $\varphi_t$ seen as a motion $\bd U_0 \to \bd U_t$ were also trivial, it would follow by patching the two trivial motions (see the proof of \thmref{uniqueness}) that $\varphi_t$ extends to a holomorphic motion $\Phi_t: \Chat \to \Chat$ which maps $U_0$ to $U_t$ and $V_0$ to $V_t$ biholomorphically for each $t$. It would easily follow that each $\Phi_t$ is a M\"{o}bius map and thus each $J_t$ is a round circle, which is a contradiction. We conclude using the Main Theorem that there is no family of Riemann maps $\DD \to U_t$ which depends holomorphically on $t$. \vs

\noindent
{\it Acknowledgements.} This is a revised version of an unpublished manuscript that I wrote in 2010. I'd like to thank Fred Gardiner and Zhe Wang for conversations at an early stage of this project. I'm specially grateful to Xavier Buff whose suggestion during a late night conversation on the Island of Porquerolles in May 2011 prompted the idea of taking Zhukovski\u{i} double-covers that is utilized in the proof of the main theorem.

\section{Preliminaries} \label{sec:prelim}

Throughout the paper we adopt the following notations: \vs
\begin{enumerate}
\item[$\bullet$]
$\Chat = \CC \cup \{ \infty \}$ \\
$\DD = \{ z \in \CC : |z|<1 \}$ \\
$\Delta=\Chat \sm \ov{\DD}$ \\
$\TT = \bd \DD = \bd \Delta = \{ z \in \CC : |z|=1 \}$. \vs
\item[$\bullet$]
$m$ is the normalized Lebesgue measure on the circle $\TT$. \vs
\item[$\bullet$]
For a distinct pair $a,b \in \TT$, $[a,b]$ is the closed arc in $\TT$
starting at $a$ and going counterclockwise to $b$. \vs
\item[$\bullet$]
We will associate various objects with a domain $U$ in $\Chat$. The
corresponding objects associated with another domain $\UT$ will be denoted by similar symbols with a tilde ``$\, \tilde{\ }\,$'' on the top, without further explanation. \vs
\end{enumerate}

We begin with a quick review of a few basic facts on conformal
mappings and Carath\'eodory's theory of prime ends. Details and proofs can be found, for example, in \cite{Mi} or \cite{Po}. \vs

By a {\bfit disk} is meant a simply connected domain $U \subset \Chat$ whose complement has more than one point. In what follows we always mark a center $c \in U$. Let $g: (\DD,0) \to (U,c)$ be a Riemann map, i.e., a conformal isomorphism $\DD \to U$ such that $g(0)=c$. The {\bfit ray} at angle $a \in \TT$ is the embedded arc
$$
\gamma(a) = \gamma_g(a) = \{ g(ra) : 0 < r < 1 \} \subset U.
$$
By the Schwarz lemma, changing $g$ to another Riemann map $(\DD,0) \to (U,c)$ will only rotate the angle of rays. We call $a$ a {\bfit landing angle} for $g$ if the radial limit
$$
g(a) = \lim_{r \to 1} g(ra) \in \bd U
$$
exists. In this case, we say that $\gamma(a)$ {\bfit lands} at $g(a)$. \vs

For each $p \in \bd U$, let $g^{-1}(p) \subset \TT$ denote
the (possibly empty) set of angles of rays in $U$ that land at $p$. The boundary
$\bd U$ decomposes into the disjoint union
\begin{equation}\label{xn}
\bd U = \bd U^0 \cup \bd U^1 \cup \bd U^2 \cup \bd U^{\geq 3},
\end{equation}
where
\begin{align*}
\bd U^n & = \{ p \in \bd U : g^{-1}(p) \ \text{has precisely} \ n \ \text{elements} \}
\qquad (0 \leq n \leq 2), \\
\bd U^{\geq 3} & = \{ p \in \bd U : g^{-1}(p) \ \text{has at least} \ 3 \ \text{elements} \}.
\end{align*}
Observe that the definition of these sets is independent of the choice of the
Riemann map $g$. Points in $\bd U \sm \bd U^0$, $\bd U^1$, and $\bd U^2$ are called
{\bfit accessible}, {\bfit uniaccessible}, and {\bfit biaccessible}, respectively. By a classical theorem of Lindel\"{o}f \cite[Corollary 2.17]{Po}, if there is a path
$\eta: [0,1) \to \DD$ with $\lim_{t \to 1} \eta(t)=a \in \TT$ and $\lim_{t \to 1}
g(\eta(t))=p \in \bd U$, then $g(a)=p$. In other words, $p \in \bd U$ is accessible
through a ray if and only if it is accessible through an arbitrary path in $U$. One immediate corollary is that the set of accessible points is always dense in $\bd U$. \vs

According to Fatou, almost every point on $\TT$ with respect to Lebesgue measure
is a landing angle. This statement can be made much sharper as follows. First, a simple topological argument proves that $\bd U^{\geq 3}$ is at most countable. Second, a theorem of F. and M. Riesz shows that $g^{-1}(p)$ has measure zero for every $p \in \bd U$ \cite[Theorem 17.4]{Mi}. Putting these facts together, we conclude that the preimage $g^{-1}(\bd U^1 \cup \bd U^2)$ already has full measure in $\TT$. In other words, almost every ray lands at a uniaccessible or biaccessible point (compare \lemref{baba} below). \vs

Another basic property that we will use is that if $p$ and $q$ are
distinct points on $\bd U$, then $g^{-1}(p)$ and $g^{-1}(q)$ are {\bfit unlinked}
subsets of $\TT$ in the sense that $g^{-1}(p)$ is contained in a single connected component of
$\TT \sm g^{-1}(q)$. This follows from the observation that distinct rays are
disjoint, hence the rays landing at $p$ cannot be separated by the rays landing at
$q$. \vs

A {\bfit cross-cut} $\eta \subset U$ is a set homeomorphic to the open interval
$(0,1) \subset \RR$ whose closure $\ov{\eta}$ is homeomorphic to the closed
interval $[0,1]$, with both endpoints on $\bd U$. The complement $U \sm \eta$
has two connected components, either of which is called a {\bfit cross-cut
domain} in $U$. A sequence $\{ O_n \}$ of cross-cut domains is a {\bfit
fundamental chain} if (i) $O_{n+1} \subset O_n$ for all $n$; (ii) the
cross-cuts $\eta_n= \bd O_n \cap U$ are pairwise disjoint; (iii) the spherical
diameter of $\eta_n$ tends to zero as $n \to \infty$. Two fundamental chains $\{
O_n \}$, $\{ O_m' \}$ are {\bfit equivalent} if each $O_n$ contains some $O'_m$
and each $O'_m$ contains some $O_n$. An equivalence class of fundamental chains
is called a {\bfit prime end} of $U$. The space of all prime ends of $U$ is
denoted by $\mathscr E$. \vs

According to Carath\'eodory, every Riemann map $g: (\DD,0) \to (U,c)$
defines a bijection $[g] : \TT \to \mathscr E$ which induces a topology and orientation on $\mathscr E$ with respect to which $[g]$ is an orientation-preserving homeomorphism. If $\{ O_n \}$ is any fundamental
chain in $U$ representing a prime end $\be \in {\mathscr E}$, then $a=[g]^{-1}(\be)$
is characterized as the unique point on the circle for which the ray $\gamma(a)$
meets every $O_n$. For $a \in \TT$, consider the {\bfit accumulation set}
$$
\Lambda(a) = \ov{\gamma(a)} \sm \gamma(a) = \bigcap_{\ve>0} \ov{ \{ g(ra): 1-\ve<r<1 \} }
$$
which is a non-empty, compact and connected subset of $\bd U$. Evidently
$\Lambda(a)$ reduces to a singleton $\{ p \}$ if and only if $\gamma(a)$ lands at
$p$. For a prime end $\be \in {\mathscr E}$, the {\bfit principal set} $\Pi(\be)$ is
the set of all $p \in \bd U$ for which there is a fundamental chain $\{ O_n \}$
representing $\be$ so that the cross-cuts $\eta_n = \bd O_n \cap U$ converge to $p$
in the spherical metric. The following basic relation between accumulation sets and principal sets holds:
\begin{equation}\label{PA1}
\Pi ([g](a)) = \Lambda(a) \qquad \text{for all} \ a \in \TT.
\end{equation}
It is well known that $g$ has a continuous extension $\ov{\DD} \to \ov{U}$ if and
only if the topological boundary $\bd U$ is locally connected. In this case, $\bd U$
is the quotient of the prime end space under the projection $\Pi: {\mathscr E} \to
\bd U$ and \eqref{PA1} reduces to $\Pi \circ [g] = g$ on $\TT$. The projection
$\Pi$, and hence the extension $g: \ov{\DD} \to \ov{U}$, is a homeomorphism if
and only if $\bd U$ is a Jordan curve. \vs

Now suppose $U, \UT \subset \Chat$ are disks with marked centers $c, \cT$ and $\varphi : \bd
U \to \bd \UT$ is a homeomorphism which is compatible with the embedding in the sphere,
that is, $\varphi$ is the restriction of an orientation-preserving homeomorphism $\Phi:
\Chat \to \Chat$ which maps $(U,c)$ to $(\UT,c)$. Evidently
$\Phi$ carries the fundamental chains in $U$ to those in $\UT$ in a one-to-one
fashion, respecting the equivalence classes. This gives an induced
homeomorphism $[\varphi]: {\mathscr E} \to \ET$ between the prime end
spaces. [Note that $[\varphi]$ depends only on $\varphi$ and not on the
particular choice of the extension $\Phi$: If $\Psi$ is another extension of $\varphi$
and $\{ O_n \}$ is a fundamental chain in $U$, then the fundamental chains $\{
O'_n = \Phi (O_n) \}$ and $\{ O''_n = \Psi(O_n) \}$ are equivalent. To see
this, fix $n$ and let $\ve$ be the spherical distance between the cross-cuts
$\bd O'_n \cap \UT$ and $\bd O'_{n+1} \cap \UT$. Since the diameter of $\bd O_m
\cap U$ tends to zero as $m \to \infty$, by uniform continuity there is an $m >
n$ so large that the cross-cut $\bd O''_m \cap \UT$ is in the
$\ve$-neighborhood of the cross-cut $\bd O'_m \cap \UT$. Then $O''_m \subset
O'_n$.] The definition of $[\varphi]$ shows that
\begin{equation}\label{PA2}
\tilde{\Pi}([\varphi](\be)) = \varphi (\Pi(\be)) \qquad \text{for all} \
\be \in {\mathscr E}.
\end{equation}
Given a pair of Riemann maps $g: (\DD,0) \to (U,c)$ and $\gT: (\DD,0) \to
(\UT,\cT)$, it follows that there is an orientation-preserving homeomorphism
$\varphi^{\circ} : \TT \to \TT$ which makes the following diagram commute:
\begin{equation}\label{PA3}
\begin{CD}
{\mathscr E} @> [\varphi] >> \ET \\
@A [g] AA  @AA [\gT] A \\
\TT @> \varphi^{\circ} >> \TT
\end{CD}
\end{equation}

\noindent
Note that changing $g$ to another Riemann map $(\DD,0) \to (U,c)$ will only
pre-compose $\varphi^{\circ}$ with a rotation. Similarly, changing $\gT$ to another
Riemann map $(\DD,0) \to (\UT,\cT)$ will only post-compose $\varphi^{\circ}$ with a
rotation. Thus, once $\varphi$ is given, $\varphi^{\circ}$ is well-defined up to pre-
and post-composition with rotations.

\begin{lemma}[Alternative characterization of $\varphi^{\circ}$]\label{EQ}
The following conditions on an orientation-preserving homeomorphism \mbox{$\sigma: \TT \to \TT$} are equivalent:
\begin{enumerate}
\item[(i)]
$\sigma=\varphi^{\circ}$, i.e., $[\gT] \circ \sigma = [\varphi] \circ [g]$ everywhere on $\TT$. \vs
\item[(ii)]
$\gT \circ \sigma = \varphi \circ g$ on the set of landing angles for $g$, hence almost everywhere on $\TT$.
\end{enumerate}
\end{lemma}

\begin{proof}
(i) $\Longrightarrow$ (ii): Suppose $a \in \TT$ is a landing angle for $g$, so the accumulation set
$\Lambda(a)$ reduces to the singleton $\{ g(a) \}$. Then by \eqref{PA1} and
\eqref{PA2},
\begin{align*}
\tilde{\Lambda}(\sigma(a)) = \tilde{\Pi}([\gT](\sigma(a))) &
= \tilde{\Pi}([\varphi]([g](a))) \\
& = \varphi(\Pi([g](a))) = \varphi(\Lambda(a)) = \{ \varphi(g(a)) \},
\end{align*}
which shows the radial limit $\gT(\sigma(a))$ exists and equals $\varphi(g(a))$. \vs

(ii) $\Longrightarrow$ (i): By the implication (i) $\Longrightarrow$ (ii) above, $\gT
\circ \varphi^{\circ} = \varphi \circ g$ on the set of landing angles for $g$. The
same argument applied to the inverse maps shows that $g \circ (\varphi^{\circ})^{-1}
= \varphi^{-1} \circ \gT$ on the set of landing angles for $\gT$. In particular,
$g(a)=p \in \bd U$ if and only if $\gT(\varphi^{\circ}(a))=\varphi(p) \in \bd \UT$.
It follows that $g^{-1}(p)$ and $\gT^{-1}(\varphi(p))$ have the same number of
elements in $[0,+\infty]$ for every $p \in \bd U$. We show that $\sigma$ and
$\varphi^{\circ}$ agree on $g^{-1}(\bd U^1 \cup \bd U^2)$ which is a full-measure set as we pointed out in \S \ref{sec:prelim}. This, by continuity, will prove  $\sigma=\varphi^{\circ}$ everywhere. \vs

Let $p \in \bd U^1 \cup \bd U^2$. The relations $\gT \circ \sigma= \gT \circ \varphi^{\circ} = \varphi \circ g$ on the set of landing angles for $g$ show that both $\sigma$ and $\varphi^{\circ}$ map
$g^{-1}(p)$ injectively into, hence bijectively onto, $\gT^{-1}(\varphi(p))$. If $p \in \bd U^1$, then $g^{-1}(p)=\{ a \}$ and $\gT^{-1}(\varphi(p))= \{ \aT \}$ for some $a,\aT \in \TT$, and
$\sigma(a)=\varphi^{\circ}(a)=\aT$ trivially. Let us then consider the case where $p
\in \bd U^2$, so $g^{-1}(p)= \{ a_1, a_2 \}$ and $\gT^{-1}(\varphi(p)) = \{ \aT_1, \aT_2 \}$
for some $a_1, a_2, \aT_1, \aT_2 \in \TT$ with $a_1 \neq a_2$ and $\aT_1 \neq \aT_2$. Take any $q
\in \bd U^1 \cup \bd U^2$ distinct from $p$. If $q \in \bd U^1$ so $g^{-1}(q)=\{ a_3 \}$ and
$\gT^{-1}(\varphi(q))=\{ \aT_3 \}$ are singletons, then both $\sigma$ and
$\varphi^{\circ}$ map $\{ a_1, a_2, a_3 \}$ to $\{ \aT_1, \aT_2, \aT_3 \}$, preserving their
cyclic order and sending $a_3$ to $\aT_3$. It follows that
$\sigma(a_1)=\varphi^{\circ}(a_1)$ and $\sigma(a_2)=\varphi^{\circ}(a_2)$. If, on the other
hand, $q \in \bd U^2$ so $g^{-1}(q)=\{ a_3, a_4 \}$ and $\gT^{-1}(\varphi(q))= \{ \aT_3, \aT_4 \}$ each have
two elements, then both $\sigma$ and $\varphi^{\circ}$ map $\{ a_1, a_2, a_3, a_4 \}$ to $\{
\aT_1, \aT_2, \aT_3, \aT_4 \}$, preserving their cyclic order, and sending the pair $\{ a_1, a_2
\}$ to $\{ \aT_1,\aT_2 \}$ and the pair $\{ a_3, a_4 \}$ to $\{ \aT_3, \aT_4 \}$. Since the
pairs $\{ a_1, a_2 \}$ and $\{ a_3, a_4 \}$ are unlinked, it follows again that
$\sigma(a_1)=\varphi^{\circ}(a_1)$ and $\sigma(a_2)=\varphi^{\circ}(a_2)$.
\end{proof}

\section{Conformal Fitness} \label{sec:coad}

We continue assuming $U, \UT \subset \Chat$ are disks with marked centers $c,\cT$ and $\varphi : \bd
U \to \bd \UT$ is a homeomorphism compatible with the embedding in the sphere.

\begin{definition}\label{CA}
We say that $\varphi: \bd U \to \bd \UT$ is {\bfit conformally fit} if there are
Riemann maps $g: (\DD,0) \to (U,c)$ and $\gT: (\DD,0) \to
(\UT,\cT)$ for which the induced homeomorphism $\varphi^{\circ}=[\gT]^{-1} \circ
[\varphi] \circ [g]: \TT \to \TT$ defined by \eqref{PA3} is the identity map. By
\lemref{EQ}, this is equivalent to the condition
$$
\gT = \varphi \circ g \quad \text{on the set of landing angles for} \ g.
$$
\end{definition}

The following (easy) theorem gives an analytic characterization of conformally fit homeomorphisms. It will be convenient to say that $\varphi : \bd U \to \bd \UT$ is the {\bfit radial limit map} of a homeomorphism
$\Phi: (U,c) \to (\UT,\cT)$ if whenever a ray $\gamma(a) \subset U$ lands at $p \in \bd
U$, the image arc $\Phi(\gamma(a)) \subset \UT$ lands at $\varphi(p) \in \bd
\UT$.

\begin{theorem}[Analytic characterization of conformal fitness] \label{adan}
The following conditions on $\varphi: \bd U \to \bd \UT$ are equivalent:
\begin{enumerate}
\item[(i)]
$\varphi$ is conformally fit. \vs
\item[(ii)]
$\varphi$ is the radial limit map of a biholomorphism $\Phi:(U,c) \to (\UT,\cT)$.
\end{enumerate}
\end{theorem}

\begin{proof}
(i) $\Longrightarrow$ (ii): Choose Riemann maps $g: (\DD,0) \to (U,c)$ and
$\gT:(\DD,0) \to (\UT,\cT)$ such that $\gT = \varphi \circ g$ on the set of landing
angles for $g$. The biholomorphism $\Phi = \gT \circ g^{-1} : (U,c) \to (\UT,\cT)$
sends rays in $U$ to rays in $\UT$, preserving their angles. It follows that
$\varphi$ is the radial limit map of $\Phi$. \vs

(ii) $\Longrightarrow$ (i): Let $\varphi$ be the radial limit map of a biholomorphism $\Phi:(U,c) \to (\UT,\cT)$, choose any Riemann map $g :(\DD,0) \to (U,c)$ and define $\gT = \Phi \circ g : (\DD,0) \to (\UT,\cT)$. Taking radial limits then shows $\gT = \varphi \circ g$ on the set of landing angles for $g$.
\end{proof}

\begin{corollary}
Suppose $\bd U, \bd \UT$ are locally connected. Then $\varphi: \bd U \to \bd
\UT$ is conformally fit if and only if there exists a biholomorphism $\Phi: (U,c) \to (\UT,\cT)$ which extends continuously to $\varphi$.
\end{corollary}

\begin{proof}
The biholomorphism $\Phi= \gT \circ g^{-1}: (U,c) \to (\UT,\cT)$ in the proof of \thmref{adan} extends continuously to $\varphi$ since by the theorem of Carath\'eodory both $g$ and $\gT$ extend continuously to the
closed unit disk, so the relation $\gT = \varphi \circ g$ holds everywhere on $\TT$. 
\end{proof}

Conformal fitness can also be characterized dynamically in terms of rotations.
Let $U \subset \Chat$ be a disk with the marked center $c$ and $g: (\DD,0) \to (U,c)$ be
a Riemann map. For each $\theta \in \RR/\ZZ$, let $\rt: (\DD,0) \to (\DD,0)$ denote the rigid rotation $z \mapsto e^{2 \pi i \theta} z$. We call the conformal automorphism $\rot = g
\circ \rt \circ g^{-1} : (U,c) \to (U,c)$ the {\bfit intrinsic rotation} of $U$ by the angle $\theta$ about $c$. By the Schwarz lemma the definition of $\rot$ is independent of the choice of $g$. \vs

\begin{theorem}[Dynamical characterization of conformal fitness]\label{adyn}
The following conditions on $\varphi: \bd U \to \bd \UT$ are equivalent:
\begin{enumerate}
\item[(i)]
$\varphi$ is conformally fit. \vs
\item[(ii)]
$\varphi$ is the radial limit map of a homeomorphism $\Phi: (U,c) \to (\UT,\cT)$
which satisfies $\Phi \circ \rot = \tilde{\rho}_{\theta} \circ \Phi$ for some irrational $\theta \in \RR / \ZZ$.
\end{enumerate}
\end{theorem}

\begin{proof}
(i) $\Longrightarrow$ (ii): Choose Riemann maps $g: (\DD,0) \to (U,c)$ and $\gT :
(\DD,0) \to (\UT,\cT)$ such that $\gT = \varphi \circ g$ on the set of landing angles
for $g$, and set $\Phi= \gT \circ g^{-1}$. Then, as in the proof of \thmref{adan},
$\varphi$ is the radial limit map of $\Phi$. Moreover, for every $\theta$,
$$
\Phi \circ \rot = \Phi \circ g \circ \rt \circ g^{-1}
= \gT \circ \rt \circ g^{-1}
= \tilde{\rho}_{\theta} \circ \gT \circ g^{-1} = \tilde{\rho}_{\theta} \circ
\Phi.
$$

(ii) $\Longrightarrow$ (i): Suppose $\varphi$ is the radial limit map of $\Phi: (U,c) \to (\UT,\cT)$ and $\Phi \circ \rot = \tilde{\rho}_{\theta} \circ \Phi$ for some irrational $\theta$. Since irrational rotations of the circle have dense orbits, it follows that the same relation holds for every $\theta$. Choose any pair of Riemann maps $g: (\DD,0) \to (U,c)$ and $\gT:(\DD,0) \to (\UT,\cT)$. The homeomorphism $\psi = \gT^{-1} \circ \Phi \circ g : (\DD,0) \to (\DD,0)$ commutes with every rigid rotation $\rt$, so it must satisfy 
\begin{equation}\label{roor}
\psi (r e^{2 \pi i t}) = \psi(r) \, e^{2 \pi i t}
\end{equation}
for every $r \in [0,1)$ and $t \in \RR$. Suppose that a ray $\gamma_g(a) \subset U$ lands at $p \in \bd U$, so
the image arc $\gamma=\Phi(\gamma_g(a))$ lands at $\varphi(p) \in \bd \UT$. By
elementary conformal mapping theory, the pull-back arc $\gT^{-1}(\gamma) = \{
\psi(ra): 0<r<1 \}$ lands at a well-defined point $b \in \TT$ \cite[Corollary
17.10]{Mi}. It follows from Lindel\"{o}f's theorem that the ray $\gamma_{\gT}(b)
\subset \UT$ lands at $\varphi(p)$. Setting $a=e^{2 \pi i t}$ in \eqref{roor} and
letting $r \to 1$, we see that $\lim_{r \to 1} \psi(r)$ exists and equals $b/a$. 
It follows that $\psi$ extends homeomorphically to the closed disk, $\psi|_{\TT}$
is the rigid rotation $z \mapsto (b/a) z$, and $\gT \circ \psi = \varphi \circ g$ on
the set of landing angles for $g$. Replacing $\gT$ by the Riemann map $z \mapsto
\gT((b/a)z)$, we obtain $\gT = \varphi \circ g$ on the set of landing angles for
$g$, which shows $\varphi$ is conformally fit.
\end{proof}

\section{Measure-Theoretic Characterization of Conformal Fitness}
\label{sec:meas}

Suppose $U \subset \Chat$ is a disk with the marked center $c$
and $g: (\DD,0) \to (U,c)$ is a Riemann map. Recall that $g(a)$
denotes the radial limit $\lim_{r \to 1} g(ra)$, which exists for Lebesgue almost
every $a \in \TT$. For $X \subset \bd U$, set
$$
g^{-1}(X) = \{ a \in \TT: g(a) \ \text{exists and belongs to} \ X \}.
$$
Define
\begin{align}\label{MM}
{\mathscr M}_B & = \{ X \subset \bd U: g^{-1}(X) \ \text{is a Borel set} \} \notag \\
{\mathscr M}_L & = \{ X \subset \bd U: g^{-1}(X) \ \text{is Lebesgue measurable} \}.
\end{align}
It is not hard to show that ${\mathscr M}_B$ and ${\mathscr M}_L$ are
$\sigma$-algebras containing all Borel subsets of $\bd U$ \cite[Proposition
6.5]{Po}. In particular, $\bd U \in {\mathscr M}_B$, so the measure-zero set
$\TT \sm g^{-1}(\bd U)$ consisting of all non-landing angles for $g$ is Borel. The
normalized Lebesgue measure $m$ on $\TT$ pushes forward under $g$ to
the probability measure $\omega = g_{\ast} \, m$ defined on ${\mathscr M}_L$,
which is called the {\bfit harmonic measure} on $\bd U$ as seen from the center $c$.
Explicitly,
$$
\omega(X) = m(g^{-1}(X)) \ \qquad \text{for} \ X \in {\mathscr M}_L.
$$
It follows from the Schwarz lemma and rotational invariance of $m$ that ${\mathscr M}_B$, ${\mathscr M}_L$ and $\omega$ do not depend
on the choice of the Riemann map $g$. \vs

The decomposition \eqref{xn} of $\bd U$ induces a decomposition
$$
\TT = A^0 \cup A^1 \cup A^2 \cup A^{\geq 3},
$$
where
\begin{align*}
A^0 & = \TT \sm g^{-1}(\bd U), \\
A^n & = g^{-1}(\bd U^n) \qquad (n=1,2), \\
A^{\geq 3} & = g^{-1}(\bd U^{\geq 3}).
\end{align*}
Note that these sets depend on $U$ as well as the choice of $g$, but
changing $g$ to another Riemann map $(\DD,0) \to (U,c)$ will only rotate them.

\begin{lemma}[Almost every point is uniaccessible or biaccessible]\label{baba}
The sets $\bd U^0, \bd U^1, \bd U^2$, and $\bd U^{\geq 3}$ belong to ${\mathscr
M}_L$, with $\omega(\bd U^0 \cup \bd U^{\geq 3})=0$. Hence,
\begin{equation}\label{bibi1}
\bd U = \bd U^1 \cup \bd U^2 \qquad \text{up to a set of harmonic measure zero}
\end{equation}
and
\begin{equation}\label{bibi2}
\TT = A^1 \cup A^2 \qquad \text{up to a set of Lebesgue measure zero}.
\end{equation}
\end{lemma}

\begin{proof}
Since $g^{-1}(\bd U^0)=\es$, we have $\bd U^0 \in {\mathscr M}_B$ and $\omega(\bd
U^0)=0$. As pointed out in \S \ref{sec:prelim}, the set $\bd U^{\geq 3}$ is at most
countable, hence Borel. Therefore, $\bd U^{\geq 3} \in {\mathscr M}_B$, which means
$A^{\geq 3}$ is Borel. Since by the theorem of F. and M. Riesz the harmonic measure
of a single point on $\bd U$ is zero, it follows that $\omega(\bd U^{\geq
3})=m(A^{\geq 3})=0$. This proves \eqref{bibi1}. The complement of $A^1 \cup A^2$ in
$\TT$ is $A^{\geq 3}$ union the set $A^0$ of all non-landing angles for $g$.
Both these sets are Borel and have Lebesgue measure zero, so $A^1 \cup A^2$ is a
Borel set of full measure on the circle. In particular, \eqref{bibi2} holds. \vs

It remains to show that $\bd U^1$ and $\bd U^2$ belong to ${\mathscr M}_L$, that is, $A^1$
and $A^2$ are Lebesgue measurable subsets of the circle. For this, observe that the set
\begin{align}\label{EE}
P & = \{ (a,b) \in (A^1 \cup A^2)^2 : g(a) = g(b) \} \sm \{ \text{diagonal} \} \\
& = \{ (a,b) \in (A^1 \cup A^2)^2 : \lim_{r \to 1} \ \text{dist}(g(ra),g(rb)) =0 \}
\sm \{ \text{diagonal} \} \notag
\end{align}
is a Borel subset of the torus $\TT \times \TT$, where ``dist'' denotes the spherical distance. Hence $A^2$, the projection of $P$ onto its first coordinate, is analytic in the sense of Suslin and therefore
measurable (recall that the projection of a Borel set need not be Borel). As $A^1
\cup A^2$ is Borel, we conclude that $A^1$ is measurable as well.
\end{proof}

Suppose now that $U, \UT \subset \Chat$ are disks with marked centers $c, \cT$ and
$\varphi: \bd U \to \bd \UT$ is a homeomorphism compatible with the embedding in the
sphere. Take any pair of Riemann maps $g: (\DD,0) \to (U,c)$ and $\gT :
(\DD,0) \to (\UT,\cT)$ and let $\varphi^{\circ} = [\gT]^{-1} \circ
[\varphi] \circ [g]: \TT \to \TT$ be the induced homeomorphism defined by
\eqref{PA3}. Consider the decompositions $\bd U = \bd U^1 \cup \bd U^2$ and $\TT =
A^1 \cup A^2$ (under $g$) as in \eqref{bibi1} and \eqref{bibi2}, and similar
decompositions $\bd \UT = \bd \UT^1 \cup \bd \UT^2$ and $\TT = \AT^1 \cup \AT^2$
(under $\gT$), all up to sets of measure zero. 

\begin{lemma}\label{pe=e}
The following diagrams are commutative:
\begin{equation}\label{twocd}
\begin{CD}
\bd U^1 @> \varphi >> \bd \UT^1 \\
@A g AA  @AA \gT A \\
A^1 @> \varphi^{\circ} >> \AT^1
\end{CD}
\qquad \text{and} \qquad
\begin{CD}
\bd U^2 @> \varphi >> \bd \UT^2 \\
@A g AA  @AA \gT A \\
A^2 @> \varphi^{\circ} >> \AT^2
\end{CD}
\end{equation}
(the vertical arrows are understood as radial limits). In particular, when $\varphi$
is conformally fit and $\gT = \varphi \circ g$ on the set of landing angles for
$g$, we have $A^1=\AT^1$ and $A^2=\AT^2$.

\end{lemma}

\begin{proof}
This follows from the fact that for $p \in \bd U$ and $a \in \TT$, $g(a)=p$ if
and only if $\gT(\varphi^{\circ}(a))= \varphi(p)$ (compare the proof of \lemref{EQ}).
\end{proof}

\begin{theorem}[Conformal fitness and harmonic measure]\label{HM1}
Suppose $U, \UT \subset \Chat$ are disks with marked centers $c,\cT$ and $\varphi: \bd U \to
\bd \UT$ is a homeomorphism compatible with the embedding in the sphere. Let
$\omega, \tilde{\omega}$ be the harmonic measures on $\bd U, \bd \UT$ as seen
from $c, \cT$.
\begin{enumerate}
\item[(i)]
If $\varphi$ is conformally fit, then $\varphi_{\ast} \, \omega =
\tilde{\omega}$. \vs
\item[(ii)]
If $\varphi_{\ast} \, \omega = \tilde{\omega}$ and if $\omega(\bd U^2)=0$, then
$\varphi$ is conformally fit.
\end{enumerate}
\end{theorem}

The assumption $\omega(\bd U^2)=0$ in part (ii) cannot be dispensed with (see \exref{dom} below). By \eqref{bibi1}, this assumption means that almost every
point of $\bd U$ is uniaccessible. As a
special case, it follows that when $\bd U, \bd\UT$ are Jordan curves, $\varphi: \bd
U \to \bd \UT$ is conformally fit if and only if $\varphi_{\ast} \, \omega =
\tilde{\omega}$.

\begin{proof}
(i) Choose Riemann maps $g:(\DD,0) \to (U,c)$ and $\gT:(\DD,0) \to (\UT,\cT)$
such that $\gT = \varphi \circ g$ on the set of landing angles for $g$, hence almost everywhere. Then,
$$
\varphi_* \, \omega = \varphi_* \, ( g_* \, m ) = (\varphi \circ g)_* \, m = \gT_* \, m = \tilde{\omega}.
$$

(ii) Choose any pair of Riemann maps $g: (\DD,0) \to (U,c)$ and $\gT:
(\DD,0) \to (\UT,\cT)$, so $\gT \circ \varphi^{\circ} = \varphi \circ g$
on the set of landing angles for $g$. Pre-composing either $g$ or $\gT$ by a rigid
rotation, we can arrange $\varphi^{\circ}(1)=1$. By \lemref{pe=e}, $\varphi(\bd U^n)
= \bd \UT^n$ for $n=1,2$. It follows from the assumption $\varphi_{\ast} \, \omega = \tilde{\omega}$ that
$$
m(\AT^2) = \tilde{\omega}(\bd \UT^2) = \omega(\varphi^{-1}(\bd \UT^2)) = \omega(\bd U^2) =0.
$$
By \eqref{bibi2}, both $A^1$ and $\AT^1$ have full measure on the
circle. The radial limit map $\gT$ sends $\AT^1$ measurably and bijectively
onto $\bd \UT^1$, and it has a measurable inverse $\gT^{-1}: \bd \UT^1 \to \AT^1$ which
pushes the harmonic measure $\tilde{\omega}$ forward to Lebesgue measure $m$
on $\AT^1$. The relation $\varphi^{\circ} = \gT^{-1} \circ \varphi \circ g$ on $A^1$, together with $\varphi_{\ast} \, \omega = \tilde{\omega}$, now imply that $\varphi^{\circ}_* \, m = m$. As a homeomorphism which preserves Lebesgue measure, $\varphi^{\circ}$ must be a rigid rotation of the circle. Since $\varphi^{\circ}$ fixes $1$, it must be the identity map.
\end{proof}

\begin{example}
Suppose $f$ is a complex monic polynomial of degree $d \geq 2$ with connected Julia
set $J$. The connected component $U$ of $\Chat \sm J$ containing $\infty$ is a disk with $\bd U =J$; it can be described as the {\bfit basin of infinity} consisting of
all points whose forward orbits under the iterations of $f$ tend to $\infty$. The
unique Riemann map $g:(\Delta,\infty) \to (U,\infty)$ which satisfies
$g'(\infty)>0$ automatically conjugates $f$ to the $d$-th power map:
\begin{equation}\label{conj}
g(z^d) = f(g(z)) \qquad \text{for all} \ z \in \Delta.
\end{equation}
In this dynamical context, $g$ is called the {\bfit B\"{o}ttcher coordinate}
for $f$. It is well-known that the ray $\gamma_g(1)$ always lands at a point $s \in J$ which by \eqref{conj} is a fixed point of $f$. The harmonic measure $\omega$ on $J$ as seen from $\infty$ (also known as the {\bfit Brolin measure} or the {\bfit
measure of maximal entropy}) is characterized as the unique ergodic measure on
$J$ which is {\bfit balanced} in the sense that $\omega(f(X))=d \cdot \omega(X)$
whenever $X \subset J$ is a Borel set on which $f$ is injective. Moreover, polynomial Julia sets have the following special property: If as before $J^2$ denotes the set of biaccessible points in $J$, then $\omega(J^2)=0$ unless $f$ is affinely conjugate to the degree $d$ Chebyshev polynomial for which $J$ is a straight line segment and $\omega(J^2)=1$ (see \cite{Sm}, \cite{Z1}, \cite{Zd}). \vs

Now suppose $f, \tilde{f}$ are two degree $d \geq 2$ non-Chebyshev polynomials with
connected Julia sets $J, \tilde{J}$ and B\"{o}ttcher coordinates $g, \gT$. By
\thmref{HM1}, a homeomorphism $\varphi: J \to \tilde{J}$ is conformally fit if
and only if $\varphi_{\ast} \, \omega = \tilde{\omega}$. This turns out to be
essentially equivalent to the condition that $\varphi$ conjugates $f$ to
$\tilde{f}$. In fact, if $\varphi$ conjugates $f$ to $\tilde{f}$, the probability
measure $\varphi_{\ast} \, \omega$ is trivially ergodic and balanced, so it must
coincide with $\tilde{\omega}$ by uniqueness. Conversely, assume $\varphi$ is
conformally fit and sends the fixed point $s=g(1)$ of $f$ to the fixed point
$\sT = \gT(1)$ of $\tilde{f}$. Choose a Riemann map $h : (\Delta, \infty) \to
(\UT,\infty)$ so the relation $h=\varphi \circ g$ holds on the set of landing angles
for $g$. In particular, $h(1)=\varphi(s)=\sT=\gT(1)$, so by the Schwarz lemma
$h=\gT$. If $p=g(a) \in J$, then
\begin{align*}
\varphi(f(p))  = \varphi(f(g(a))) & = \varphi(g(a^d)) = \gT(a^d) \\
 & = \tilde{f}(\gT(a)) = \tilde{f}(\varphi(g(a))) = \tilde{f}(\varphi(p)).
\end{align*}
In other words, $\varphi$ conjugates $f$ to $\tilde{f}$ on the set of accessible points in
$J$. Since this set is dense, continuity of $\varphi$ shows that $\varphi \circ
f = \tilde{f} \circ \varphi$ everywhere on $J$. Note that the assumption
$\varphi(s)=\sT$ in this argument is essential: The involution
$\varphi(z)=-z$ on the Julia set of any quadratic polynomial $f(z)=z^2+c$ is
conformally fit, but clearly it is not a self-conjugacy of $f$.
\end{example}

%%%%%%%%%%%%%%%%%%%%%%%%%%%%%%
\realfig{slit}{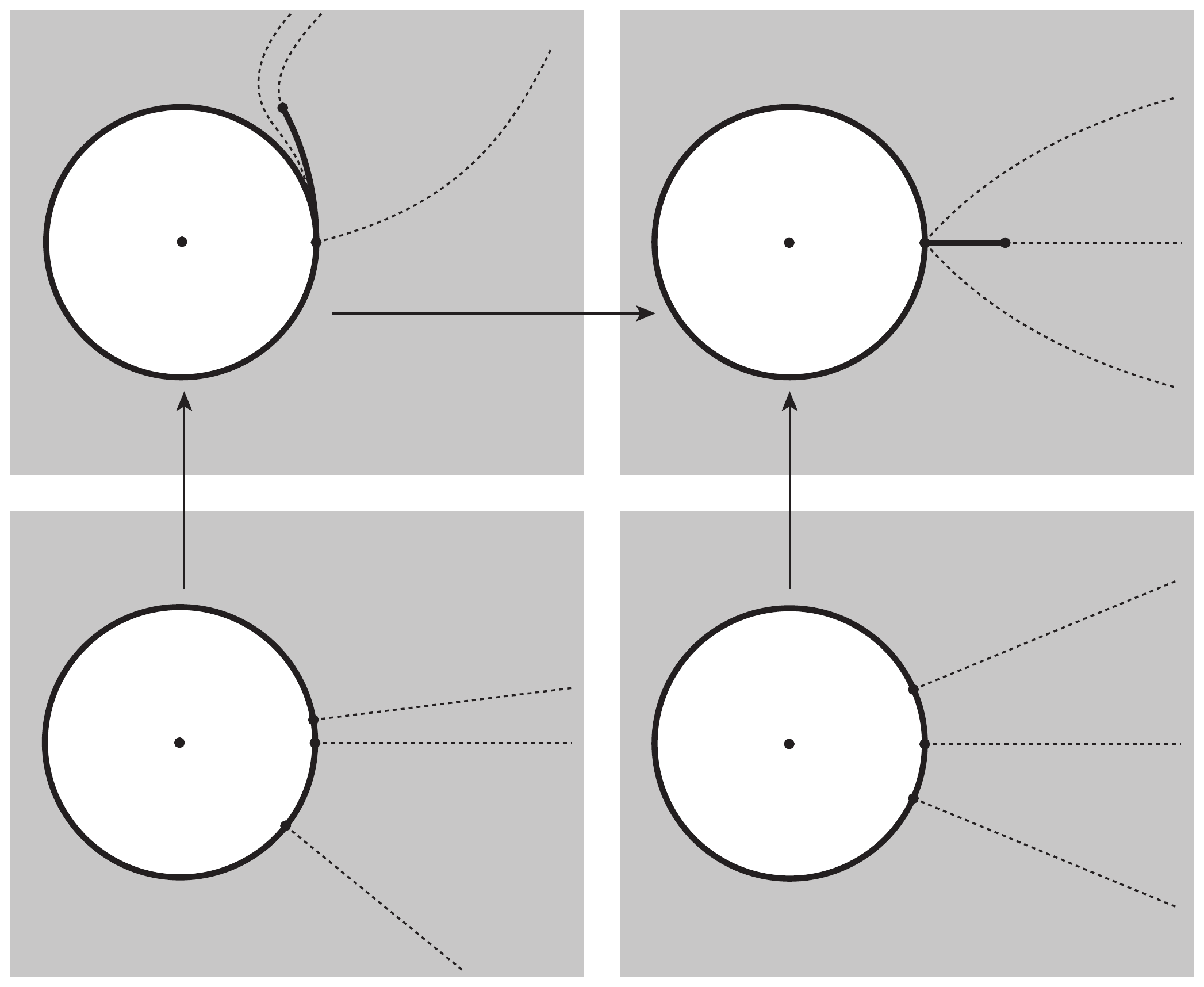}{{\sl Illustration of \exref{dom}: There is a homeomorphism
$\varphi : \bd U \to \bd \UT$ which respects harmonic measures but fails to be
conformally fit. The arcs $\ell, \tilde{\ell}$ have equal harmonic measures in
their respective domains as seen from $\infty$. The two sides of $\tilde{\ell}$ receive the same
harmonic mass because of symmetry, but one side of $\ell$, being more exposed to
rays, receives more harmonic mass than the other side.}
$$
 \at{-4.1\pscm}{11.9\pscm}{$\ell$}
 \at{2.4\pscm}{11.4\pscm}{$\tilde{\ell}$}
 \at{-7.0\pscm}{13.1\pscm}{$U$}
 \at{-0.9\pscm}{13.1\pscm}{$\UT$}
 \at{-7.0\pscm}{8.1\pscm}{$\Delta$}
 \at{-0.9\pscm}{8.1\pscm}{$\Delta$}
 \at{-2.5\pscm}{10.8\pscm}{$\varphi$}
 \at{-5.6\pscm}{5.9\pscm}{$0$}
 \at{0.5\pscm}{5.9\pscm}{$0$}
 \at{-5.6\pscm}{10.9\pscm}{$0$}
 \at{0.5\pscm}{10.9\pscm}{$0$}
 \at{-4.4\pscm}{11.2\pscm}{$1$}
 \at{1.7\pscm}{11.2\pscm}{$1$}
 \at{-4.4\pscm}{6.2\pscm}{$1$}
 \at{1.7\pscm}{6.2\pscm}{$1$}
 \at{-4.35\pscm}{12.7\pscm}{$p$}
 \at{2.7\pscm}{10.9\pscm}{$\pT$}
 \at{-5.3\pscm}{8.75\pscm}{$g$}
 \at{0.8\pscm}{8.75\pscm}{$\tilde{g}$}
 \at{-4.5\pscm}{5.15\pscm}{$a$}
 \at{-4.1\pscm}{6.65\pscm}{$b$}
 \at{1.8\pscm}{5.35\pscm}{$\aT$}
 \at{1.8\pscm}{7.05\pscm}{$\bT$}
$$
}
{12cm}
%%%%%%%%%%%%%%%%%%%%%%%%%%%%%%

\begin{example}\label{dom}
Fix $e^{2\pi i \delta} \in \TT$ for some $0<\delta<1/2$, and pick a small $\ve>0$. Let $\ell$ be the shorter of the two closed arcs with endpoints $1$ and $p=(1+\ve)e^{2\pi i \delta}$ on the circle centered on the negative real axis. Consider the disk $U=\Delta \sm \ell$ and the Riemann map $g: (\Delta,\infty) \to (U,\infty)$ normalized so that
$g(1)=p$. Since $\bd U = \TT \cup \ell$ is locally connected, $g$ extends
continuously to a map between the closures, and $g^{-1}(\ell)$ a closed arc $X
\subset \TT$ with endpoints $a,b$ containing $1$ in its interior, labeled as in \figref{slit}. Furthermore, $g$
maps $\TT \sm X$ homeomorphically to $\bd U \sm \ell$, and $X \sm \{ 1 \}$
two-to-one to $\ell \sm \{ p \}$. It is not hard to check that $a \to e^{-2 \pi i \delta}$ and $b \to 1$ as $\ve \to 0$, so the harmonic measure $\omega(\ell)=m(X)$ tends to $\delta$. In particular, taking $\ve >0$ sufficiently small guarantees 
that $b$ is closer to $1$ than $a$ is. We also consider a symmetric
disk $\UT = \Delta \sm \tilde{\ell}$, where $\tilde{\ell}$ is the closed segment
from $1$ to a real point $\tilde{p}>1$, and the Riemann map $\gT: (\Delta,\infty)
\to (\UT,\infty)$ normalized so that $\gT(1)=\tilde{p}$. Since
$\tilde{\omega}(\tilde{\ell})$ increases continuously from $0$ to $1$ as $\tilde{p}$ goes from $1$ to $\infty$, we can choose $\tilde{p}$ so that
$\tilde{\omega}(\tilde{\ell})=\omega(\ell)>0$. Again by local connectivity, $\gT$
extends continuously to a map between the closures, but this time $\gT^{-1}(\tilde{\ell})$ is a closed symmetric arc $\tilde{X} \subset \TT$ with endpoints $\aT,\bT$ having $1$ as its midpoint. As before, $\gT$ maps $\TT \sm \tilde{X}$
homeomorphically to $\bd \UT \sm \tilde{\ell}$, and $\tilde{X} \sm \{ 1 \}$
two-to-one to $\tilde{\ell} \sm \{ \tilde{p} \}$. \vs

Define $\varphi: \bd U \sm \ell \to \bd \UT \sm \tilde{\ell}$ by $\varphi = \gT
\circ g^{-1}$. Let $\ell_x$ (resp. $\tilde{\ell}_x$) denote the closed subarc of
$\ell$ (resp. $\tilde{\ell}$) from $1$ to $x \in \ell$ (resp. $x \in \tilde{\ell}$).
The function $\mu: x \mapsto \omega (\ell_x)$ is continuous on $\ell$ and
monotonically increases from $0$ to $\omega(\ell)$ as $x$ goes from $1$ to $p$.
Similarly, $\tilde{\mu}: x \mapsto \tilde{\omega} (\tilde{\ell}_x)$ is continuous on
$\tilde{\ell}$ and monotonically increases from $0$ to
$\tilde{\omega}(\tilde{\ell})=\omega(\ell)$ as $x$ goes from $1$ to $\tilde{p}$.
Define $\varphi : \ell \to \tilde{\ell}$ by $\varphi = \tilde{\mu}^{-1} \circ \mu$.
The map $\varphi : \bd U \to \bd \UT$ constructed this way is a homeomorphism which
satisfies $\varphi_* \ \omega = \tilde{\omega}$. However, $\varphi$ is not
conformally fit since otherwise $\gT = \varphi \circ g$ would hold everywhere on $\TT$ and in particular $a=\aT,b=\bT$, which is a contradiction.
\end{example}

Thus, respecting the harmonic measure is generally a weaker condition than conformal fitness. To arrive at a measure-theoretic characterization of conformal fitness that applies to arbitrary disks, we need a refinement of the harmonic
measure which takes into account the mass concentrated on different ``sides'' of
biaccessible points of the boundary when approached from within the disk. \vs

It will be convenient to choose coordinates on $\Chat$ so the disk $U$ has its marked center at $\infty$. The boundary $\bd U$ is then a non-empty compact connected subset of the plane. Mark a boundary point $s \in \bd U^1$ and let $g: (\Delta,\infty) \to (U,\infty)$ be the unique Riemann map with $g(1)=s$.
For each $p \in \bd U^2$ let $g^{-1}(p)= \{ a^-, a^+ \}$, where the angles $a^-=a^-(p)$ and
$a^+=a^+(p)$ are labeled so that the interval $[a^-,a^+] \subset \TT$ does not contain $1$, which means the three angles appear in the cyclic order  $1 \to a^- \to a^+$ as we
go counter-clockwise around the circle (see \figref{rays}). \vs

%%%%%%%%%%%%%%%%%%%%%%%%%%%%%%
\realfig{rays}{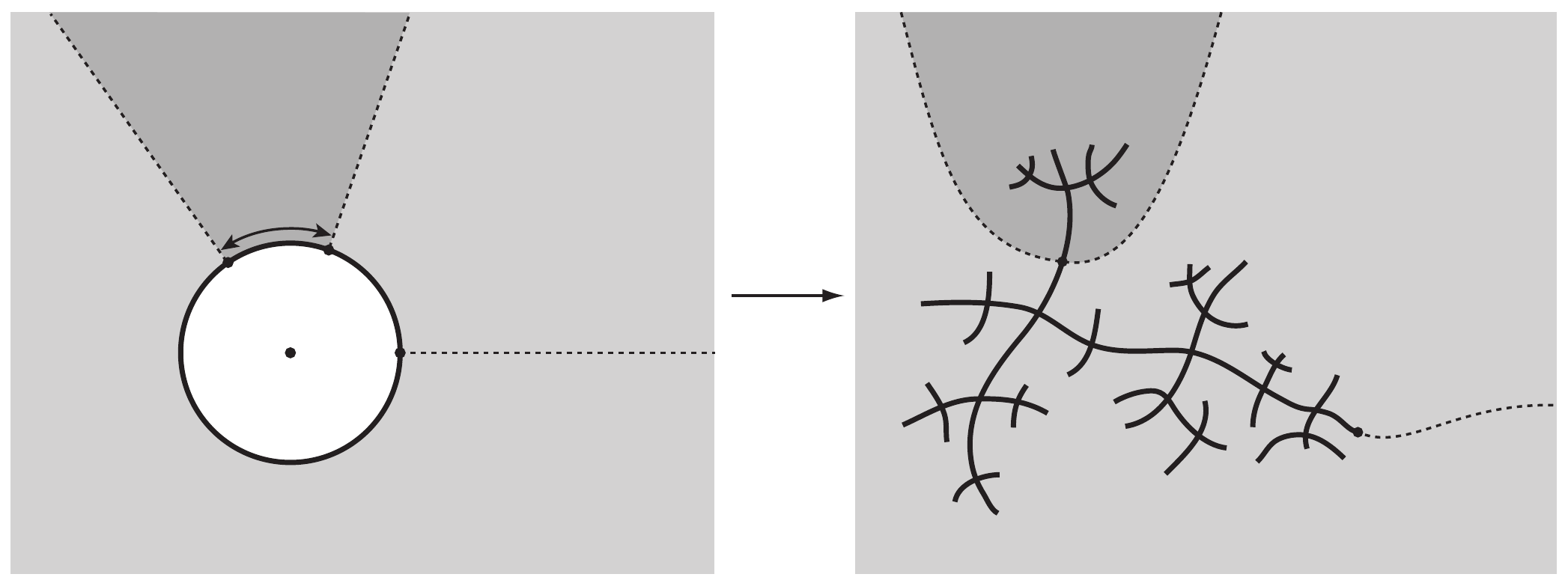}{{\sl Labeling the two preimages of a biaccessible point.}
$$
 \at{-6.1\pscm}{3.25\pscm}{{\footnotesize $0$}}
 \at{-7.1\pscm}{4.35\pscm}{{\footnotesize $a^{\! +}$}}
 \at{-5.5\pscm}{4.65\pscm}{{\footnotesize $a^-$}}
 \at{-2.6\pscm}{6.2\pscm}{$\Delta$}
 \at{-4.9\pscm}{3.25\pscm}{{\small $1$}}
 \at{-1.4\pscm}{4.35\pscm}{$g$}
 \at{5.4\pscm}{6.2\pscm}{$U$}
 \at{1.4\pscm}{4.25\pscm}{{\small $p$}}
 \at{4.2\pscm}{3.0\pscm}{{\small $s$}}
 \at{5.2\pscm}{2.75\pscm}{{\footnotesize $\gamma(1)$}}
 \at{-0.5\pscm}{5.05\pscm}{{\footnotesize $\gamma(a^{\! +})$}}
 \at{2.6\pscm}{5.6\pscm}{{\footnotesize $\gamma(a^-)$}}
$$}
{15cm}
%%%%%%%%%%%%%%%%%%%%%%%%%%%%%%

This labeling allows us to decompose the set $A^2=g^{-1}(\bd U^2)$ of biaccessible
angles in $\TT$ into the disjoint subsets
$$
A^{2-} = \{ a^-(p) : p \in \bd U^2 \} \qquad \text{and} \qquad A^{2+} = \{ a^+(p) : p \in \bd U^2 \}.
$$
By \eqref{bibi2},
\begin{equation}\label{decomT}
\TT = A^1 \cup A^{2-} \cup A^{2+} \qquad \text{up to a set of Lebesgue measure zero}.
\end{equation}
This decomposition is not canonical (even modulo a rotation) since it depends on the choice of the marked boundary point $s$. In fact, it is easy to see that there can be no canonical way of distinguishing the two rays landing at a biaccessible point. If we replace $s$ with another marked boundary point $s'$, the senses $\pm$ get reversed precisely on those ray pairs which separate $s$ from $s'$ (compare \figref{biacc}). \vs

To verify measurability of $A^{2 \pm}$, consider the Borel set $P$ defined in (\ref{EE}). Then
$A^{2-}$ and $A^{2+}$ are the projections of the Borel set
$$
P \cap \{ (a,b) \in \TT \times \TT : [1,a] \subset [1,b] \}
$$
onto the first and second coordinates respectively, hence they are measurable.
\vs

The following is a refinement of (part of) \lemref{pe=e}.

\begin{lemma}\label{pe=e1}
Suppose the homeomorphism $\varphi: \bd U \to \bd \UT$ respects the boundary markings so
$\varphi(s)=\sT$, and $g: (\Delta,\infty) \to (U,\infty)$ and $\gT : (\Delta,
\infty) \to (\UT,\infty)$ are the unique Riemann maps normalized by $g(1)=s$ and
$\gT(1)=\tilde{s}$. Then the second commutative diagram in \eqref{twocd} splits into
a pair of commutative diagrams
$$
\begin{CD}
\bd U^2 @> \varphi >> \bd \UT^2 \\
@A g AA  @AA \gT A \\
A^{2\pm} @> \varphi^{\circ} >> \AT^{2\pm}
\end{CD}
$$
In particular, if $p \in \bd U^2$ and $\tilde{p}=\varphi(p)$, then $\varphi^{\circ}$
maps the interval $[a^-(p),a^+(p)]$ homeomorphically onto the interval
$[a^-(\tilde{p}),a^+(\tilde{p})]$.
\end{lemma}

%%%%%%%%%%%%%%%%%%%%%%%%%%%%%%
\realfig{biacc}{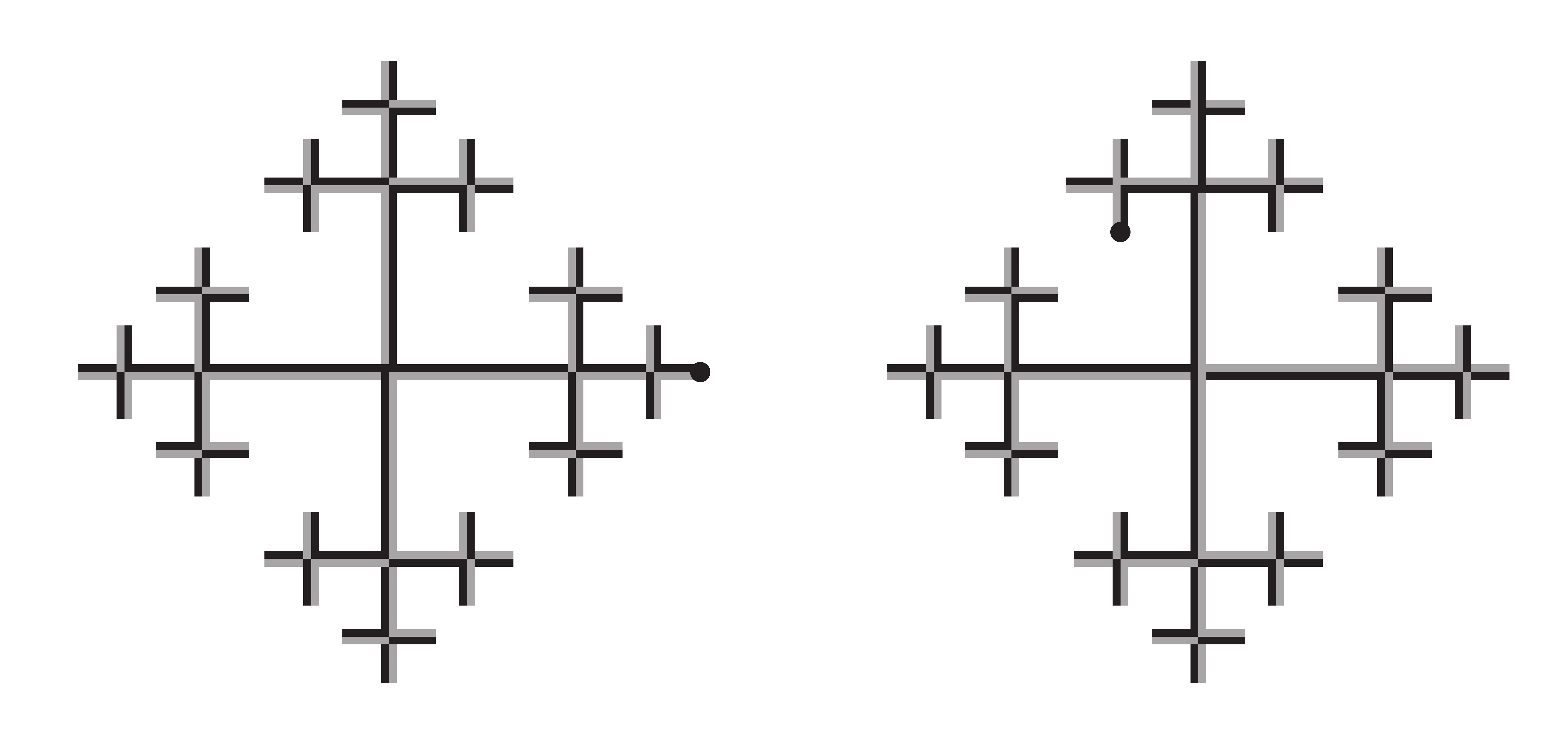}{{\sl The senses $\pm$ for the biaccessible points on the
boundary of a disk centered at $\infty$ with a marked boundary point $s$, and the
effect of changing $s$. The black side corresponds to $-$ and the gray side
corresponds to $+$. Changing the location of $s$ reverses the $\pm$ sense on ray pairs which separate the old and new marked points.}
$$
 \at{-1.9\pscm}{6.75\pscm}{$s$}
 \at{1.9\pscm}{7.9\pscm}{$s$}
$$
}{15cm}
%%%%%%%%%%%%%%%%%%%%%%%%%%%%%%

\begin{proof}
If $p \in \bd U^2$, then $\tilde{p}=\varphi(p) \in \bd \UT^2$ by \lemref{pe=e}.
The relation $\gT \circ \varphi^{\circ} = \varphi \circ g$ on $A^2$ shows that
$\varphi^{\circ} (\{ a^-(p) , a^+(p) \}) = \{ a^-(\tilde{p}) , a^+(\tilde{p}) \}$.
Since $\varphi^{\circ}: \TT \to \TT$ is orientation-preserving and fixes $1$, we
must have $\varphi^{\circ} (a^{\pm}(p))=a^{\pm}(\tilde{p})$.
\end{proof}

Take a disk $U$ centered at $\infty$ with a marked boundary point $s \in \bd U^1$, and take the unique Riemann map $g:(\Delta,\infty) \to (U,\infty)$ with $g(1)=s$. Define the measures $\alpha, \beta^{\pm}$ on the $\sigma$-algebra ${\mathscr M}_L$ of \eqref{MM} by
\begin{align}\label{ab}
\alpha(X) & = m(g^{-1}(X) \cap A^1)= \omega(X \cap \bd U^1), \notag \\
\beta^{\pm}(X) & = m(g^{-1}(X) \cap A^{2\pm}).
\end{align}
Note that $\alpha$ is independent of the choice of the marked point $s$, but the measures $\beta^{\pm}$
certainly depend on it. This gives a decomposition of the harmonic measure into the
sum
$$
\omega = \alpha + \beta^- + \beta^+.
$$

The following is a generalization of \thmref{HM1}:

\begin{theorem}[Measure-theoretic characterization of conformal fitness]\label{mtadapt}
Suppose $U, \UT \subset \Chat$ are disks with marked centers at $\infty$ and marked
boundary points $s, \sT$, and $\varphi: \bd U \to \bd \UT$ is a homeomorphism
compatible with the embedding in the sphere such that $\varphi(s)=\sT$. Let $\omega = \alpha + \beta^- + \beta^+$ and $\tilde{\omega} = \tilde{\alpha} + \tilde{\beta}^- + \tilde{\beta}^+$ be the decompositions of the harmonic measures on $\bd U$ and $\bd \UT$ as constructed above. Then the following
conditions are equivalent: \vs
\begin{enumerate}
\item[(i)]
$\varphi$ is conformally fit. \vs
\item[(ii)]
$\varphi_{\ast} \, \alpha = \tilde{\alpha}$ and $\varphi_{\ast} \, \beta^{\pm} = \tilde{\beta}^{\pm}$.
\end{enumerate}
\end{theorem}

\begin{proof}
(i) $\Longrightarrow$ (ii): Take the Riemann maps $g:(\Delta,\infty) \to (U,
\infty)$ and $\gT:(\Delta,\infty) \to (\UT,\infty)$ with $g(1)=s$ and $\gT(1)=\sT$. By conformal fitness, $\gT = \varphi \circ g$ on the set of landing angles for $g$ and in particular on $A^1 \cup A^{2-} \cup A^{2+}$. By \lemref{pe=e1}, $A^1=\AT^1$ and $A^{2\pm}=\AT^{2\pm}$. Hence, for every set $X \subset \bd \UT$ in the $\sigma$-algebra $\tilde{\mathscr M}_L$,
\begin{align*}
(\varphi_{\ast} \, \alpha)(X) = \alpha (\varphi^{-1}(X)) & = m( g^{-1}(\varphi^{-1}(X)) \cap A^1 ) \\
& = m (\gT^{-1}(X) \cap \AT^1) = \tilde{\alpha}(X).
\end{align*}
One verifies $\varphi_{\ast} \, \beta^{\pm} = \tilde{\beta}^{\pm}$ similarly. \vs

(ii) $\Longrightarrow$ (i): Let $g, \gT$ be as above. By \lemref{EQ} and
$\gT(1)=\varphi(g(1))$, the induced homeomorphism $\varphi^{\circ} = [\gT]^{-1}
\circ [\varphi] \circ [g] : \TT \to \TT$ fixes $1$. Take any $a \in \TT \sm \{ 1 \}$
and let $X = g([1,a]) \subset \bd U$. If $p \in X \cap \bd U^2$ and $a^+(p) \in
[1,a]$, then necessarily $a^-(p) \in [1,a]$. Thus, up to a
set of harmonic measure zero, $X$ decomposes into the sets
\begin{align*}
X_1 & = X \cap \bd U^1 \\
X_2 & = \{ p \in X \cap \bd U^2 : \ \text{both} \ a^{\pm}(p) \ \text{are in} \ [1,a] \} \\
X_3 & = \{ p \in X \cap \bd U^2 : \ \text{only} \ a^-(p) \ \text{is in} \ [1,a] \}.
\end{align*}
We have $X_1 \in {\mathscr M}_L$ since $g^{-1}(X_1)=[1,a] \cap A^1$ is measurable.
Also, $X_2 \in {\mathscr M}_L$ since $g^{-1}(X_2)$ is the projection onto the first
coordinate of the Borel set $P \cap ([1,a]\times[1,a])$, where $P$ is defined in \eqref{EE}.
Finally, $X_3 \in {\mathscr M}_L$ since $g^{-1}(X_3)$ is the union of the
projections onto the first and second coordinates of the Borel set $P \cap ([1,a]
\times (a , 1))$. \vs

This gives a measurable decomposition $[1,a] = I_1 \cup I^-_2 \cup I^+_2 \cup I_3$, up to a set of Lebesgue measure zero, where
\begin{align*}
I_1 & = g^{-1}(X_1) = [1,a] \cap A^1 \\
I^-_2 & = g^{-1}(X_2) \cap A^{2-} \\
I^+_2 & = g^{-1}(X_2) \cap A^{2+} \\
I_3 & = g^{-1}(X_3) \cap A^{2-}.
\end{align*}
Setting $\XT = \varphi(X) = \gT([1,\varphi^{\circ}(a)]) \subset \bd \UT$, we arrive at similar decompositions
$\XT = \XT_1 \cup \XT_2 \cup \XT_3$ and $[1,\varphi^{\circ}(a)] = \IT_1 \cup \IT^-_2
\cup \IT^+_2 \cup \IT_3$ up to sets of measure zero. Note that by \lemref{pe=e}
and \lemref{pe=e1}, $\varphi(X_n)=\XT_n$ for $n=1,2,3$. The assumptions
$\varphi_{\ast} \, \alpha = \tilde{\alpha}$ and $\varphi_{\ast} \, \beta^{\pm} =
\tilde{\beta}^{\pm}$ now show that
\begin{align*}
m(\IT_1) & = \tilde{\alpha}(\XT_1) = \alpha(X_1) = m(I_1) \\
m(\IT_2^{\pm}) & = \tilde{\beta}^{\pm}(\XT_2) = \beta^{\pm}(X_2) =
m(I_2^{\pm}) \\
m(\IT_3) & = \tilde{\beta}^- (\XT_3) = \beta^-(X_3) = m(I_3).
\end{align*}
Adding up these equalities gives $m([1,\varphi^{\circ}(a)])=m([1,a])$. Since
$\varphi^{\circ}(1)=1$, we must have $\varphi^{\circ}(a)=a$. Since $a$ was
arbitrary, we conclude that $\varphi^{\circ}=\text{id}$.
\end{proof}

\section{The Zhukovski\u{i} lift} \label{sec:zdc}

This section describes a topological procedure that turns some
biaccessible points on the boundary of a disk into uniaccessible points. It will be used
in the next section where we study conformal fitness in a holomorphically
moving family of disks. Consider the {\bfit Zhukovski\u{i} map} $Z: \Chat \to \Chat$ defined by 
$$
Z(\zeta) = \zeta + \frac{1}{\zeta}.
$$
This is a degree $2$ rational map of the sphere with critical points at $\pm 1$
and critical values at $Z(\pm 1)=\pm 2$. It is easily checked that $Z$ maps the
disks $\DD$ and $\Delta=\Chat \sm \ov{\DD}$ biholomorphically onto the slit-sphere $\Chat
\sm [-2,2]$, with $Z(0)=Z(\infty)=\infty$, $Z'(\infty)=1$. Note also the symmetry
$Z(\zeta)=Z(1/\zeta)$. \vs

Suppose $U$ is a disk with the marked center at $\infty$ such that $\pm 2 \in
\bd U$. Since $U$ is simply connected and does not contain the critical values of
$Z$, it lifts under $Z$ to a disk $V$ centered at $\infty$ and the map $Z: V \to U$
is a biholomorphism (the other lift of $U$ is the image of $V$ under the
inversion $\zeta \mapsto 1/\zeta$). Evidently, $Z$ extends to a continuous surjection $\ov{V} \to \ov{U}$. We call $V$ the {\bfit
Zhukovski\u{i} preimage} of $U$. Note that since $Z'(\infty)=1$, this process does not
alter the conformal radius:
\begin{equation}\label{r=r}
\rad(V,\infty)= \rad(U,\infty).
\end{equation}

Assuming further that $2 \in \bd U^1$, we can designate it as the marked boundary
point of $U$. Let $g:(\Delta,\infty) \to (U, \infty)$ be the unique Riemann map with
$g(1)=2$. The composition $h=Z^{-1}\circ g:(\Delta,\infty) \to (V,\infty)$ is then
the Riemann map for $V$ which satisfies $h(1)=1$. The relation $g=Z \circ h$ in $\Delta$
persists under radial limits, hence $g(a)$ exists for some $a \in \TT$ if and only
if $h(a)$ exists, and $g(a)=Z(h(a))$. In other words, the ray $\gamma_g(a) \subset
U$ lifts under $Z$ to the ray $\gamma_h(a) \subset V$, hence $\gamma_g(a)$ lands at $p
\in \bd U$ if and only if $\gamma_h(a)$ lands at some $q \in \bd V \cap Z^{-1}(p)$.
It follows in particular that $1 \in \bd V^1$, so we can designate it as the marked
boundary point of $V$. With these boundary markings on $\bd U$ and $\bd V$, we can speak of the senses $\pm$ on the set of
biaccessible angles for $g$ and $h$, as described in \S \ref{sec:meas}. \vs

Let $p=g(a^-)=g(a^+)$ be a biaccessible point on $\bd U$. It will be convenient to say that the {\bfit ray pair} $\gamma_g(a^\pm)$ {\bfit separates} $z$ {\bfit from} $w$ if $z$ and $w$ belong to different connected components of $\CC \sm (\gamma_g(a^-) \cup \gamma_g(a^+) \cup \{ p \})$.

\begin{lemma}[Lifting accessible points]
\mbox{}
\begin{enumerate}
\item[(i)]
Suppose the ray $\gamma_g(a)$ lands at a uniaccessible point $p \in \bd U$. Then the lifted ray $\gamma_h(a)$ lands at a uniaccessible point which is the unique $Z$-preimage of $p$ on $\bd V$. \vs
\item[(ii)]
Suppose the rays $\gamma_g(a^-), \gamma_g(a^+)$ land at a biaccessible point $p \in \bd U \sm \{ 2,-2 \}$. If the ray pair $\gamma_g(a^\pm)$ does not separate $-2$ from $2$, the lifted rays $\gamma_h(a^-), \gamma_h(a^+)$ land at a biaccessible point which is the unique $Z$-preimage of $p$ on $\bd V$. If the ray pair $\gamma_g(a^\pm)$ separates $-2$ from $2$, the lifted rays $\gamma_h(a^-), \gamma_h(a^+)$ land at a pair of uniaccessible points which are the distinct $Z$-preimages of $p$ on $\bd V$.
\end{enumerate}
\end{lemma}

\figref{doublecover} illustrates this lemma.

%%%%%%%%%%%%%%%%%%%%%%%%%%%%%%
\realfig{doublecover}{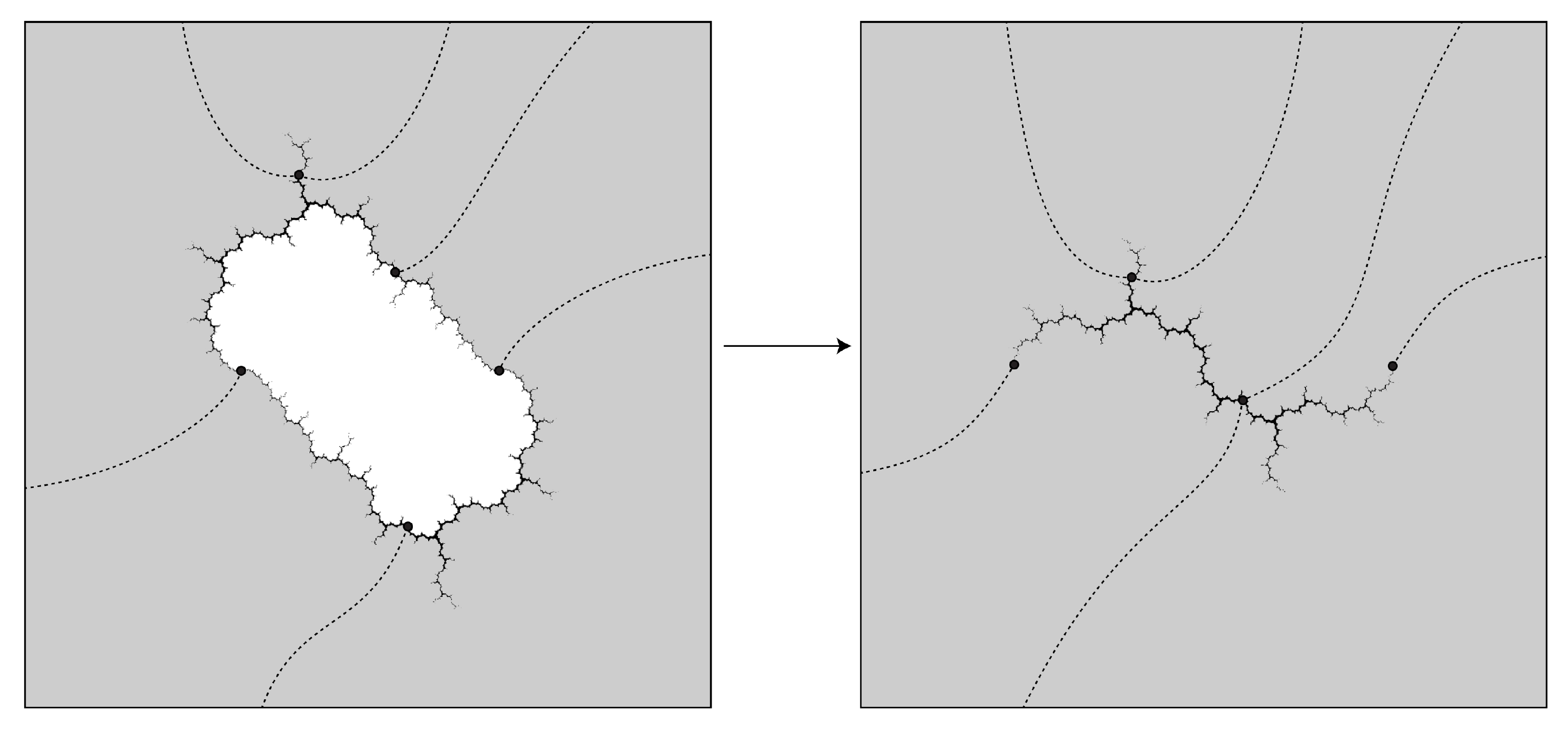}{{\sl A disk $U$ centered at $\infty$ and its
Zhukovski\u{i} preimage $V$. Each uniaccessible point on $\bd U$ (such as $2$ or $-2$
here) lifts to a unique uniaccessible point on $\bd V$. But a biaccessible point on
$\bd U$ lifts to a pair of uniaccessible points or a unique biaccessible point on
$\bd V$ according as the corresponding ray pair is or is not separating $-2$ from
$2$. (For simplicity, only the angles of rays with respect to the normalized Riemann
maps are shown in the figure.)}
$$
 \at{-0.3\pscm}{10.7\pscm}{$U$}
 \at{-8.3\pscm}{10.7\pscm}{$V$}
 \at{4.7\pscm}{7.9\pscm}{{\footnotesize $2$}}
 \at{0.4\pscm}{7.9\pscm}{{\footnotesize $-2$}}
 \at{-4.1\pscm}{7.6\pscm}{{\footnotesize $1$}}
 \at{-6.6\pscm}{8.0\pscm}{{\footnotesize $-1$}}
 \at{-1.4\pscm}{8.3\pscm}{$Z$}
 \at{5.2\pscm}{11.0\pscm}{{\footnotesize $a^-$}}
 \at{0.7\pscm}{4.8\pscm}{{\footnotesize$a^+$}}
 \at{3.2\pscm}{10.9\pscm}{{\footnotesize $b^-$}}
 \at{\pscm}{10.9\pscm}{{\footnotesize$b^+$}}
 \at{5.8\pscm}{9.1\pscm}{{\footnotesize$1$}}
 \at{-0.4\pscm}{6.7\pscm}{{\footnotesize$c$}}
 \at{-3.2\pscm}{10.9\pscm}{{\footnotesize $a^-$}}
 \at{-6.6\pscm}{4.8\pscm}{{\footnotesize$a^+$}}
 \at{-5.0\pscm}{10.9\pscm}{{\footnotesize $b^-$}}
 \at{-6.8\pscm}{10.9\pscm}{{\footnotesize$b^+$}}
 \at{-2.2\pscm}{9.15\pscm}{{\footnotesize$1$}}
 \at{-8.4\pscm}{6.6\pscm}{{\footnotesize$c$}}
$$
}{15cm}
%%%%%%%%%%%%%%%%%%%%%%%%%%%%%%

\begin{proof}
(i) We have $p=g(a) \in \bd U^1$. Let $q=h(a)$ so $q \in \bd V$ and $Z(q)=p$. If $q=h(b)$
for some $b \neq a$, then $g(b)=Z(h(b))=Z(q)=p$, which contradicts $p$ being
uniaccessible. Thus, $q \in \bd V^1$. To prove $Z^{-1}(p) \cap \bd V = \{ q \}$, we may assume $p \neq \pm 2$ so $Z^{-1}(p)$ consists of distinct points $q, 1/q$, each having a small neighborhood which maps biholomorphically under $Z$ to a small
neighborhood of $p$. If $1/q \in \bd V$ also, it follows that $1/q$ is
accessible, so by Lindel\"{o}f's theorem $1/q=h(b)$ for some $b \neq a$. Applying
$Z$ gives $g(b)=p$, which once again contradicts $p$ being uniaccessible. \vs

(ii) If $p=g(a^-)=g(a^+) \in \bd U^2$, then $h(a^-)$ and $h(a^+)$ both belong to $\bd
V$ and map under $Z$ to $p$. First suppose the ray pair $\gamma_g(a^\pm)$ does not separate $-2$ from $2$ and choose an embedded arc $\eta$ from $-2$ to $2$ which avoids the union $\gamma_g(a^-) \cup \gamma_g(a^+)
\cup \{ p \}$ entirely. The disk $\Chat \sm \eta$ lifts under $Z$ to a disk $D$
containing $\infty$, so the map $Z: D \to \Chat \sm \eta$ is a biholomorphism.
Evidently, the preimages of $\gamma_g(a^{\pm})$ under this biholomorphism
are the rays $\gamma_h(a^{\pm})$, which in particular shows that the latter rays have a common landing point. It follows from a similar argument as above that this common landing point is in $\bd V^2$ and forms $Z^{-1}(p) \cap \bd V$. Next, suppose the ray pair $\gamma_g(a^\pm)$ separates $-2$ from $2$. Then every embedded arc from $-2$ to $2$ which avoids the rays $\gamma_g(a^{\pm})$ must pass through $p$. Let $\eta$ be such an arc and define $D$ as before and note that $D$ is bounded by the
Jordan curve $Z^{-1}(\eta)$ passing through $\pm 1$. Each of the two connected components of $\CC \sm (\gamma_g(a^-) \cup \gamma_g(a^+) \cup \eta)$ has either $2$ or $-2$ on its boundary. Taking preimages under the biholomorphism $Z: D \to \Chat \sm \eta$, we see that each of the two connected components of $(D \sm \infty) \sm (\gamma_h(a^-) \cup \gamma_h(a^+))$ has either $1$ or $-1$ on its boundary.  It easily follows that the rays $\gamma_h(a^{\pm})$ must land at distinct $Z$-preimages of $p$. These points are uniaccessible since otherwise $p$ would be the landing point of $3$ or more rays.
\end{proof}

\begin{corollary}\label{joojoo}
Let $V$ be the Zhukovski\u{i} preimage of a disk centered at $\infty$. Suppose $q \in \bd V \sm \{ -1, 1 \}$ is biaccessible and the rays $\gamma_h(a^\pm)$ land at $q$. Then the ray pair $\gamma_h(a^\pm)$ does not separate $-1$ from $1$.
\end{corollary}

\begin{lemma}\label{zlift}
Suppose $U$ and $\UT$ are disks with marked centers at $\infty$, with $\pm 2$ on $\bd
U^1$ and $\bd \UT^1$, and with Zhukovski\u{i} preimages $V$ and $\VT$. Let $\varphi: \bd U \to
\bd \UT$ be a homeomorphism which fixes $\pm 2$ and is compatible with the embedding in the sphere. Then there is a unique
homeomorphism $\psi: \bd V \to \bd \VT$ which fixes $\pm 1$, is compatible with the embedding in the sphere, and makes the following diagram commute:
$$
\begin{CD}
\bd V @> \psi >> \bd \VT \\
@V Z VV  @VV Z V \\
\bd U @> \varphi >> \bd \UT
\end{CD}
$$
Moreover, for any pair of Riemann maps $g:(\Delta, \infty) \to (U, \infty)$, $\gT:(\Delta, \infty) \to (\UT, \infty)$ and $h: (\Delta,\infty) \to (V,\infty)$, $\tilde{h}: (\Delta,\infty) \to (\VT,\infty)$ related by $g=Z \circ h$ and $\gT=Z \circ \tilde{h}$, the induced circle homeomorphisms $\varphi^{\circ}=[\gT]^{-1} \circ [\varphi] \circ [g]$ and $\psi^{\circ}=[\tilde{h}]^{-1} \circ [\psi] \circ [h]$ coincide. In particular, $\psi$ is conformally fit if and only if $\varphi$ is.
\end{lemma}

We call $\psi$ the {\bfit Zhukovski\u{i} lift} of $\varphi$.

\begin{proof}
Extend $\varphi$ to an orientation-preserving homeomorphism $\Phi: (\Chat, \infty) \to (\Chat, \infty)$ fixing $\pm 2$. Both $Z$ and $\Phi \circ Z: \Chat \sm  \{ \pm 1 \} \to \Chat \sm  \{ \pm 2 \}$ are degree $2$ regular coverings. The image of the infinite cyclic fundamental group $\pi_1(\Chat \sm \{ \pm 1 \}, \infty)$ under $\Phi \circ Z$ is evidently the same as its image under $Z$. Hence there is a unique lift $\Psi: \Chat \sm \{ \pm 1 \} \to \Chat \sm \{ \pm 1 \}$ which fixes $\infty$ and satisfies $\Phi \circ Z=Z \circ \Psi$. Applying the same argument to $\Phi^{-1}$ shows that $\Psi$ has a continuous inverse, and hence it extends to a homeomorphism $\Psi : (\Chat, \infty) \to (\Chat, \infty)$ once we define $\Psi(\pm 1) = \pm 1$. Since $\Psi(V)=\VT$ by the construction, the restriction $\psi=\Psi|_{\bd V}$ is the desired homeomorphism. \vs

To verify the last assertion, note that the Riemann maps $g$ and $h$, as well as $\gT$ and $\tilde{h}$, have identical landing angles,  and the relations $g=Z \circ h$ and $\gT=Z \circ \tilde{h}$ persist under radial limits. It follows that on the set of landing angles for $g$,   
$$
\gT \circ \psi^{\circ} = Z \circ \tilde{h} \circ \psi^{\circ} = Z \circ \psi \circ h = \varphi \circ Z \circ h = \varphi \circ g. 
$$
\lemref{EQ} now shows that $\varphi^{\circ}=\psi^{\circ}$. 
\end{proof}

In the situation of the above lemma, let $\omega, \tilde{\omega}$ denote the harmonic measures on $\bd U, \bd \UT$ and $\nu, \tilde{\nu}$ denote the harmonic measures on $\bd V, \bd \VT$, all as seen from $\infty$. Clearly, $\omega = Z_{\ast} \, \nu$ and $\tilde{\omega}=Z_{\ast} \, \tilde{\nu}$. If $\psi_{\ast} \, \nu = \tilde{\nu}$, the functorial property of the push-forward operator easily gives $\varphi_{\ast} \, \omega = \tilde{\omega}$. However, the relation $\varphi_{\ast} \, \omega = \tilde{\omega}$ does not necessarily imply $\psi_{\ast} \, \nu = \tilde{\nu}$. This can be seen from a suitably normalized version of \exref{dom} in which $p$ and $\tilde{p}$ are placed at $2$ and, say, $0$ is placed at $-2$. {\it Thus, the Zhukovski\u{i} lift of a homeomorphism which respects the harmonic measure may fail to do so}.

\section{Holomorphically Moving Disks} \label{sec:hm}

Suppose $\{ U_t \}_{t \in \DD}$ is a family of domains in $\Chat$ with the property that
\begin{equation}\label{op}
\{ (z,t) : z \in U_t \} \subset \Chat \times \DD \quad \text{is open}.
\end{equation}
Equivalently, for every $t_0 \in \DD$ and every compact set $K \subset U_{t_0}$, the inclusion
$K \subset U_t$ holds for all $t$ sufficiently close to $t_0$. We say that a family
of maps $f_t : U_t \to \Chat$ {\bfit depends holomorphically on} $t$ if
whenever $z_0 \in U_{t_0}$, the map $t \mapsto f_t(z_0)$ is holomorphic in some
neighborhood of $t_0$. \vs

We include the following easy result for convenience:

\begin{lemma}\label{easy}
\mbox{}
\begin{enumerate}
\item[(i)]
Let $\{ U_t \}_{t \in \DD}$ and $\{ V_t \}_{t \in \DD}$ be families of
domains in $\Chat$ that satisfy \eqref{op}. Suppose $f_t : U_t \to \Chat$
and $g_t : V_t \to \Chat$ are families of maps depending holomorphically on $t$ such that $g_t(V_t) \subset U_t$ and $f_t$ is holomorphic for each $t$. Then the composition $f_t \circ g_t$ depends holomorphically on $t$. \vs
\item[(ii)]
If, in addition, each $f_t$ is univalent, then the inverse family $f_t^{-1}:
f_t(U_t) \to \Chat$ depends holomorphically on $t$.
\end{enumerate}
\end{lemma}

\begin{proof}
(i) Let $U=\{ (z,t) : z \in U_t \}$. By a classical theorem of Hartogs \cite{Na}, $f_t : U_t \to \Chat$ depends holomorphically on $t$ if and only if the map $f: U \to \Chat$ defined by $f(z,t)=f_t(z)$ is holomorphic as a function of two complex variables. Thus, for each $z_0 \in V_{t_0}$ the map $t \mapsto (f_t \circ g_t)(z_0)$ is holomorphic in a neighborhood of $t_0$ since it is the composition of the holomorphic map $t \mapsto (g_t(z_0),t)$ followed by $f$. \vs

(ii) Let $W_t=f_t(U_t)$ and $W = \{ (z,t) : z \in W_t \}$ and note
that by Hartogs' theorem, the map $F : U \to W$ defined by $F(z,t)=(f(z,t),t)$ is injective and holomorphic. Hence $W$ is open and the inverse $F^{-1}$ is holomorphic. This, by another application of Hartogs, shows that $f_t^{-1}:W_t \to \Chat$ depends holomorphically on $t$.
\end{proof}

A {\bfit holomorphic motion} of a set $X_0 \subset \Chat$ over the
unit disk $\DD$ is a family of injections $\varphi_t : X_0 \hookrightarrow \Chat$ which depends holomorphically on the parameter $t \in \DD$ and reduces to the identity map at $t=0$. Explicitly,
\begin{enumerate}
\item[(i)]
for each $t \in \DD$, the map $z \mapsto \varphi_t(z)$ is injective on
$X_0$;
\item[(ii)]
for each $z \in X_0$, the map $t \mapsto \varphi_t(z)$ is holomorphic in
$\DD$;
\item[(iii)]
for each $z \in X_0$, $\varphi_0(z)=z$.
\end{enumerate}
We often write $X_t$ for the image $\varphi_t(X_0)$. \vs

The following fundamental results on extending holomorphic motions will be used in this section. The first is a rather easy application of Montel's theorem on normal families:

\begin{theorem}[The $\lambda$-Lemma]\label{ll}
Every holomorphic motion $\varphi_t : X_0 \to X_t$ over $\DD$ has a unique extension to a holomorphic motion $\ov{X}_0 \to \ov{X}_t$ of the closure.
\end{theorem}

The second is a much deeper result with a more difficult proof \cite{Sl}:

\begin{theorem}[Slodkowski] \label{slod}
Suppose $\varphi_t : X_0 \to X_t$ is a holomorphic motion of a set $X_0 \subset
\Chat$ over $\DD$. Then there is a holomorphic motion $\Phi_t: \Chat \to
\Chat$ of the Riemann sphere such that $\Phi_t|_{X_0} = \varphi_t|_{X_0}$
for all $t \in \DD$. 
\end{theorem}

It is an elementary fact that every holomorphic motion $\Phi_t: \Chat \to
\Chat$ is automatically $K_t$-quasiconformal, where $K_t=(1+|t|)/(1-|t|)$. \vs

As a simple application, we record the following not-so-obvious uniqueness result:

\begin{theorem} \label{uniqueness}
A given family $\{ X_t \}_{t \in \DD}$ of sets in $\Chat$ with empty interior is the image of at most one holomorphic motion over $\DD$: if $\varphi_t,\psi_t : X_0 \to X_t$ are both holomorphic motions, then
$\varphi_t=\psi_t$ for all $t \in \DD$.
\end{theorem}

\begin{proof}
Use Slodkowski's \thmref{slod} to extend $\varphi_t$ and $\psi_t$ to
holomorphic motions $\Phi_t$ and $\Psi_t$ of the whole sphere. The map
$$
z \mapsto
\begin{cases}
\Phi_t(z) & z \in X_0 \\
\Psi_t(z) & z \in \Chat \sm X_0
\end{cases}
$$
is then a holomorphic motion of the sphere. By continuity, $\varphi_t(z)=\psi_t(z)$ for all $z \in X_0$.
\end{proof}

This section will study families $\{ U_t \}$ of disks in $\Chat$ and holomorphic motions $\varphi_t : \bd U_0 \to \bd U_t$ over $\DD$ of their boundaries. Such motions impose a form of rigidity on the boundaries since by Slodkowski's theorem they are all quasiconformally equivalent. For instance, if $\bd U_t$ is a quasicircle for some $t$, the same must hold for all $t$. Moreover, for any family of Riemann maps $g_t: \DD \to U_t$ the induced homeomorphisms $\varphi^{\circ}_t = [g_t]^{-1} \circ [\varphi_t] \circ [g_0] : \TT \to \TT$ must be quasisymmetric. To see this, extend $\varphi_t$ to a holomorphic motion $\Phi_t : \Chat \to \Chat$ which takes $U_0$ to $U_t$. The composition $h_t = g_t^{-1} \circ \Phi_t \circ g_0 : \DD \to \DD$ is then quasiconformal for each $t$, so it extends to a quasisymmetric homeomorphism $h_t: \TT \to \TT$. Writing $g_t \circ h_t = \Phi_t \circ g_0$ in $\DD$ and taking radial limits gives $g_t \circ
h_t = \varphi_t \circ g_0$ on the set of landing angles for $g_0$. It follows from
\lemref{EQ} that $h_t=\varphi^{\circ}_t$. \vs

Despite this rigidity, holomorphic motions are flexible and can be easily produced by invoking the measurable Riemann mapping theorem of Morrey-Ahlfors-Bers \cite{Ah}. For any holomorphic motion $\Phi_t: \Chat \to \Chat$ over $\DD$, the {\bfit Beltrami coefficient} $\mu_t = \ov{\bd} \Phi_t / \bd \Phi_t \in L^{\infty}(\Chat)$ is well-defined and has the following properties: (i) $\| \mu_t \|_{\infty} <1$; (ii) $\mu_0=0$; (iii) $\mu_t$ depends holomorphically on the parameter $t \in \DD$. Conversely, any family $\{ \mu_t \}_{t \in \DD}$ of Beltrami coefficients on $\Chat$ which satisfies the above three conditions comes from a holomorphic motion. In fact, let $\Phi_t: \Chat \to \Chat$ be the unique quasiconformal homeomorphism which solves the Beltrami equation $\ov{\bd} \Phi_t / \bd \Phi_t = \mu_t$ and is normalized so that it fixes $0,1,\infty$. According to Ahlfors and Bers, for each $z \in \Chat$ the map $t \mapsto \Phi_t(z)$ is holomorphic. Since $\Phi_0=\text{id}$ by uniqueness, it follows that $\Phi_t$ defines a holomorphic motion of the sphere over $\DD$. This provides a convenient way of embedding any disk boundary $\bd U_0$ in a holomorphic motion: Take any measurable function $\mu$ on $\Chat$ such that $\| \mu \|_{\infty} <1$, consider the holomorphic motion $\Phi_t: \Chat \to \Chat$ generated by the Beltrami coefficients $\mu_t=t \mu$, and set $U_t=\Phi_t(U_0)$. The restriction of $\Phi_t$ to the boundary $\bd U_0$ is then a holomorphic motion $\bd U_0 \to \bd U_t$ over $\DD$. \vs

Let $\{ (U_t,c_t) \}_{t \in \DD}$ be a family of pointed disks in $\Chat$, where the marked center $c_t \in U_t$ depends holomorphically on $t$. We say that a holomorphic motion $\varphi_t: \bd U_0 \to \bd U_t$ is {\bfit trivial} if it extends to a holomorphic motion $\Phi_t: (\Chat,c_0) \to (\Chat,c_t)$ whose Beltrami coefficient $\mu_t$ is supported off $U_0$. In other words, for each $t \in \DD$ the restriction $\Phi_t : (U_0,c_0) \to (U_t,c_t)$ should be a biholomorphism. Trivial motions of a disk boundary $\bd U_0$ are easily produced: Simply choose the initial Beltrami coefficient $\mu$ in the above construction to be supported on $\Chat \sm U_0$, so every $\mu_t=t \mu$ vanishes in $U_0$. \vs

We are now ready to state our main theorem which ties in several of the notions introduced earlier: 

\begin{namedtheorem}[Main]
Let $\{ (U_t, c_t) \}_{t \in \DD}$ be a family of pointed disks in $\Chat$, where the center $c_t$ depends holomorphically on $t$. Suppose the boundaries of these disks undergo a holomorphic motion $\varphi_t: \bd U_0 \to \bd U_t$ over $\DD$. Then the following conditions are equivalent: \vs
\begin{enumerate}
\item[(i)]
There is a family of Riemann maps $g_t : (\DD,0) \to (U_t,c_t)$ which depends
holomorphically on $t$. \vs
\item[(ii)]
$\varphi_t: \bd U_0 \to \bd U_t$ is a trivial motion in the sense that it extends to a holomorphic motion $\Phi_t: (\Chat, c_0) \to (\Chat,c_t)$ such that the restriction $\Phi_t: (U_0,c_0) \to (U_t,c_t)$ is a biholomorphism for each $t$. \vs
\item[(iii)]
$\varphi_t : \bd U_0 \to \bd U_t$ is conformally fit for each $t$. \vs
\item[(iv)]
There is an irrational $\theta \in \RR/\ZZ$ for which the intrinsic rotation $\rott: (U_t,c_t) \to
(U_t,c_t)$ depends holomorphically on $t$. \vs
\item[(v)]
$\varphi_t: \bd U_0 \to \bd U_t$ respects the harmonic measure as seen from the center: $(\varphi_t)_{\ast} \omega_0 = \omega_t$ for each $t$. \vs
\item[(vi)]
The map $t \mapsto \log \rad(U_t,c_t)$ is harmonic in $\DD$. 
\end{enumerate}
\end{namedtheorem}

Before we begin the proof, several remarks are in order: \vs

$\bullet$ The family $\{ g_t \}$ in (i), if exists, must be unique up to a rotation and therefore is completely determined by the choice of $g_0:(\DD,0) \to (U_0,c_0)$. In fact, if $g_t, h_t : (\DD,0) \to (U_t,c_t)$ are both Riemann maps which depend holomorphically on $t$, the composition $h_t^{-1} \circ g_t : (\DD,0) \to (\DD,0)$ is a rigid rotation by the Schwarz lemma, which  depends holomorphically on $t$ by \lemref{easy}, hence must be independent of $t$. \vs

$\bullet$ It is easy to verify that the conditions (i)-(vi) in the above theorem are invariant under M\"{o}bius change of coordinates: If $\{ M_t \}_{t \in \DD}$ is a family of M\"{o}bius maps which depends holomorphically on $t$, then by \lemref{easy} the map $\psi_t = M_t \circ \varphi_t \circ M_0^{-1}: \bd V_0 \to \bd V_t$ defines a holomorphic motion on the boundaries of the disks $V_t=M_t(U_t)$. Moreover, each of the conditions (i)-(vi) holds for $\varphi_t$ if and only if the corresponding condition holds for $\psi_t$. As an example, take two distinct points $p_0, q_0$ on $\bd U_0$, let
$p_t=\varphi_t(p_0)$ and $q_t=\varphi_t(q_0)$, and consider the unique M\"{o}bius map $M_t: \Chat \to
\Chat$ which carries $(p_t,q_t,c_t)$ to $(0,1,\infty)$. The the above construction will then produce a holomorphic motion where the pointed disks have their marked center at $\infty$. \vs

$\bullet$ A holomorphic motion $\varphi_t: \bd U_0 \to \bd U_t$ for which $t \mapsto \log \rad(U_t,c_t)$ is harmonic can always be rescaled so the conformal radius becomes constant. To see this, apply a preliminary M\"{o}bius change of coordinates to arrange $c_t=\infty$ for all $t$. Then find a holomorphic function $f: \DD \to \CC$ such that $\Real (f(t)) = \log \rad(U_t,\infty)$ for all $t \in \DD$ and let $M_t: \Chat \to \Chat$ be the dilation $z \mapsto e^{-f(t)} z$. The boundaries of the rescaled disks $V_t = M_t(U_t)$ undergo the holomorphic motion $M_t \circ \varphi_t \circ M_0^{-1}$ and the conformal radius of the pointed disk $(V_t,\infty)$ is $1$ for all $t$.

\begin{proof}[Proof of the Main Theorem]
We will verify the following implications:

$$
\begin{tabular}{ccccccccc}
(iv) & $\Longleftrightarrow$ & (i) & $\Longrightarrow$ & (ii) & & & & \vs \\
& & & \rotatebox[origin=c]{135}{$\Longrightarrow$} & $\Downarrow$ & & & & \vs \\
& & & & (iii) & $\Longleftrightarrow$ & (v) & $\Longleftrightarrow$ & (vi) \vs
\end{tabular}
$$

\noindent
As noted above, we may assume without loss of generality that $c_t=\infty$ for all $t$. \vs

(i) $\Longrightarrow$ (ii): Use Slodkowski's \thmref{slod} to extend $\varphi_t$ to
a holomorphic motion $\psi_t: (\Chat,\infty) \to (\Chat,\infty)$ which necessarily
takes $U_0$ to $U_t$. The map $\Phi_t: (\Chat,\infty) \to (\Chat,\infty)$ defined by
\begin{equation}\label{P}
\Phi_t = \begin{cases} g_t \circ g_0^{-1} & \text{in} \ U_0 \\ \psi_t & \text{on}
\ \Chat \sm U_0 \end{cases}
\end{equation}
is a holomorphic motion of the sphere which extends $\varphi_t$, fixes $\infty$, and
maps $U_0$ biholomorphically to $U_t$. \vs

(ii) $\Longrightarrow$ (iii): Extend $\varphi_t$ to a holomorphic motion $\Phi_t: (\Chat,\infty) \to (\Chat,\infty)$ which maps $U_0$ biholomorphically to $U_t$. For any Riemann map $g_0: (\Delta, \infty) \to (U_0,\infty)$,
the composition $g_t = \Phi_t \circ g_0 : (\Delta, \infty) \to (U_t,\infty)$ is a conformal isomorphism. Taking
radial limits, it follows that $g_t = \varphi_t \circ g_0$ on the set of
landing angles for $g_0$. \vs

(iii) $\Longrightarrow$ (i): Let the Riemann maps $g_t: (\Delta, \infty) \to
(U_t,\infty)$ be chosen so that for each $t$ the relation $g_t = \varphi_t \circ
g_0$ holds on the set of landing angles for $g_0$. The holomorphic dependence of $g_t$ on $t$ is then a straightforward consequence of the Poisson integral formula once we normalize maps and disks properly.
Pick distinct points $p_{1,0}, p_{2,0}, p_{3,0}$ on $\bd U_0$, let
$p_{n,t}=\varphi_t(p_{n,0})$ for $n=1,2,3$, and consider the unique M\"{o}bius map $M_t: \Chat \to
\Chat$ which carries $(p_{1,t}, p_{2,t}, p_{3,t})$ to $(0,1,\infty)$. The composition
$\psi_t = M_t \circ \varphi_t \circ M_0^{-1}$ gives a holomorphic motion of the
disk $V_t=M_t(U_t) \subset \Chat \sm \{ 0, 1, \infty \}$ centered at $c_t =
M_t(\infty)$. The Riemann maps $h_t: (\DD,0) \to (V_t, c_t)$ defined by $h_t(z)= (M_t
\circ g_t)(1/z)$ satisfy $h_t=\psi_t \circ h_0$ on the set of landing angles for $h_0$. It
suffices to show that $h_t$ depends holomorphically on $t$. \vs

To this end, let $\pi: \DD \to \Chat \sm \{ 0, 1, \infty \}$ be a holomorphic
universal covering map. Let $t \mapsto \hat{c}_t$ be any lift of the holomorphic map $t \mapsto c_t$
under $\pi$ and $\hat{h}_t : (\DD, 0) \to (\DD, \hat{c}_t)$ be the corresponding
unique lift of $h_t$. Ignoring the measure-zero set of angles on $\TT$ for which
the radial limit of $h_0$ does not exist or belongs to $\{ 0, 1, \infty \}$, the relation $h_t=\psi_t \circ h_0$ shows that $t \mapsto
h_t(a)$ is holomorphic and takes values in $\Chat \sm \{ 0, 1, \infty \}$ for almost every $a \in \TT$. Lifting under $\pi$, it follows that the radial limit $t \mapsto \hat{h}_t(a)$ is holomorphic and takes values in $\DD$ for almost every $a \in \TT$. As $\hat{h}_t \in L^\infty(\TT)$, the Poisson formula
$$
\hat{h}_t (z) = \frac{1}{2\pi} \int_0^{2\pi} \frac{1-|z|^2}{|e^{is}-z|^2} \, \hat{h}_t(e^{is}) \, ds \qquad (z \in \DD)
$$
holds. Since the integrand depends holomorphically on $t$ for almost every $s$, the map $t \mapsto
\hat{h}_t(z)$ is holomorphic for each $z \in \DD$, so the same must be true of $t
\mapsto h_t(z)=\pi(\hat{h}_t(z))$. \vs

(i) $\Longrightarrow$ (iv): This follows from \lemref{easy} since $\rott = g_t \circ \rt \circ
g_t^{-1}$. \vs

(iv) $\Longrightarrow$ (i): This is essentially Sullivan's argument in \cite{Su}. Use Slodkowski's \thmref{slod} to extend $\varphi_t$ to a holomorphic motion $\psi_t: (\Chat,\infty) \to (\Chat,\infty)$ which maps $U_0$ to $U_t$. Take a sequence $\{ z_{n,0} \}$ in $U_0$ such that $z_{n,0} \to \bd U_0$
as $n \to \infty$. Set $z_{n,t} = \psi_t(z_{n,0}) \in U_t$, so $z_{n,t} \to \bd U_t$
as $n \to \infty$. Let $\Gamma_{n,t} \subset U_t$ be the Jordan curve through $z_{n,t}$
which is invariant under the intrinsic rotation $\rott$. By the assumption (iv) and
\lemref{easy}, the map
$$
\rho_{\theta,0}^{\circ k}(z_{n,0}) \mapsto \rott^{\circ k}(z_{n,t}) \qquad k=1,2,3,\ldots
$$
defines a holomorphic motion of the $\rho_{\theta,0}$-orbit of $z_{n,0}$. By the
$\lambda$-Lemma and irrationality of $\theta$, this motion extends to a holomorphic
motion $\zeta_{n,t} : \Gamma_{n,0} \to \Gamma_{n,t}$. Let $g_{n,t}$ be the unique
Riemann map from $\Delta$ to the connected component of $\Chat \sm
\Gamma_{n,t}$ containing $\infty$, normalized so that
$g_{n,t}(\infty)=\infty$ and $g_{n,t}(1)=z_{n,t}$. It follows from the definition of
the intrinsic rotations that $g_{n,t} = \zeta_{n,t} \circ g_{n,0}$ on the unit circle
$\TT$. In particular, $t \mapsto g_{n,t}(z)$ is holomorphic if $z \in \TT$. An application of the Poisson formula (similar to but easier than the above argument) then shows that $t \mapsto g_{n,t}(z)$ is holomorphic for $z \in \Delta$. For each $t \in \DD$ the
family $\{ g_{n,t} \}_{n \geq 1}$ is normal in $\Delta$ since it omits every value
in $\Chat \sm U_t$, and any limit of it as $n \to \infty$ is a Riemann map $(\Delta,
\infty) \to (U_t,\infty)$. We need to choose these limits consistently to
guarantee they depend holomorphically on $t$. To this end, take any sequence $t_k$
of parameters tending to $0$ and by a diagonal argument find an increasing sequence
$\{ n_j \}$ of integers such that $g_{n_j,t_k}$ converges locally uniformly in
$\Delta$ for each $k$ as $j \to \infty$. Pick distinct points $p_{1,0}, p_{2,0}, p_{3,0}$ on $\bd U_0$ so the three holomorphic maps $t \mapsto \varphi_t(p_{i,0})$ for $i=1,2,3$ have disjoint graphs. Montel's theorem shows that for each $z \in \Delta$ the
family $\{ t \mapsto g_{n,t}(z) \}_{n \geq 1}$ is normal in $\DD$ since their graphs
do not intersect those of $t \mapsto \varphi_t(p_{i,0})$. It follows from
Vitali-Porter's Theorem \cite{Re} that $t \mapsto g_{n_j,t}(z)$ converges locally
uniformly in $\DD$ for every $z \in \Delta$ as $j \to \infty$. Clearly the limit
map $g_t= \lim_{j \to \infty} g_{n_j,t}$ depends holomorphically on $t$. \vs

(iii) $\Longrightarrow$ (v): This follows immediately from \thmref{HM1}. \vs

The proof of the equivalence (v) $\Longleftrightarrow$ (vi) will depend on a
potential theoretic characterization of the harmonic measure which we recall
below (for details, see for example \cite{Ra}). Suppose $U \subset \Chat$ is a disk
containing $\infty$, so $\bd U$ is a compact subset of the plane. The {\bfit energy}
of a Borel probability measure $\mu$ supported on $\bd U$ is defined as the integral
\begin{equation}\label{energy}
E(\mu) = - \iint_{\bd U \times \bd U} \log | z - w | \, d\mu(z) \, d\mu(w).
\end{equation}
The harmonic measure $\omega$ on $\bd U$ as seen from $\infty$ is the unique
{\bfit equilibrium measure} on $\bd U$, i.e., the unique measure which
minimizes energy among all Borel probability measures on $\bd U$ \cite[Theorem 4.3.14]{Ra}:
\begin{equation}\label{min}
E(\omega) = \inf_{\mu} E(\mu)
\end{equation} 
A Green's function argument gives a geometric interpretation for the minimal
energy by relating it to the conformal radius \cite[Theorem 5.2.1]{Ra}:
\begin{equation}\label{e/r}
E(\omega) = \log \rad(U,\infty).
\end{equation}

(v) $\Longrightarrow$ (vi): Since $(\varphi_t)_* \omega_0 = \omega_t$, the
definition \eqref{energy} shows that
$$
E(\omega_t) = - \iint_{\bd U_0 \times \bd U_0}
\log |\varphi_t(z)-\varphi_t(w)| \, d\omega_0(z) \, d\omega_0(w).
$$
The integrand on the right is a harmonic function of $t$, so the same must be
true of $E(\omega_t)$. Since $\log \rad(U_t,\infty) = E(\omega_t)$ by \eqref{e/r}, we
conclude that $t \mapsto \log \rad(U_t,\infty)$ is harmonic in $\DD$. \vs

(vi) $\Longrightarrow$ (v): Consider the push-forward measure $\mu_t=(\varphi_t)_* \omega_0$ on $\bd U_t$ and define
\begin{align*}
h(t) = E(\mu_t) & = - \iint_{\bd U_t \times \bd U_t} \log |z-w| \, d\mu_t(z) \, d\mu_t(w) \\
& = - \iint_{\bd U_0 \times \bd U_0} \log |\varphi_t(z)-\varphi_t(w)| \, d\mu_0(z) \, d\mu_0(w).
\end{align*}
The function $h$ is harmonic in $\DD$ with $h(0)=E(\omega_0)$ since
$\mu_0=\omega_0$. Furthermore, the energy minimizing characterization
\eqref{min} of $\omega_t$ shows that $h(t) \geq E(\omega_t)$ for every $t \in
\DD$. The non-negative harmonic function $t \mapsto h(t) - E(\omega_t)$ assumes the value $0$ at $t=0$, hence must be identically $0$ by the maximum
principle. The uniqueness of the equilibrium measure then shows that
$\mu_t=\omega_t$ for all $t$. \vs

(v) $\Longrightarrow$ (iii): Choose a marked point $s_0 \in \bd U_0^1$ and set
$s_t=\varphi_t(s_0) \in \bd U_t^1$. Let $g_t: (\Delta, \infty) \to (U_t, \infty)$ be
the unique Riemann map which satisfies $g_t(1)=s_t$, and 
$\varphi^{\circ}_t=[g_t]^{-1} \circ [\varphi_t] \circ [g_0]: \TT \to \TT$ be the
induced circle homeomorphism, which fixes $1$ by our normalization. We prove that
$\varphi^{\circ}_t=\text{id}$ for all $t$. \vs

Let $A_t^1$ and $A_t^2$ denote the set of uniaccessible and biaccessible angles for $\bd U_t$ under $g_t$, as defined in \S \ref{sec:meas}. Fix an angle $a_0 \in A_0^1 \cup A_0^2$, $a_0 \neq 1$, and set $q_0=g_0(a_0)$. Then $q_t=\varphi_t(q_0)=g_t(a_t)$, where $a_t=\varphi^{\circ}_t(a_0) \in A_t^1 \cup
A_t^2$. Let $M_t: \Chat \to \Chat$ be the affine map which sends $2$ to
$s_t$ and $-2$ to $q_t$:
$$
M_t: z \mapsto \frac{1}{4}(s_t-q_t) z + \frac{1}{2}(s_t+q_t).
$$
Let $W_t = M_t^{-1}(U_t)$ and consider the induced motion $\zeta_t= M_t^{-1} \circ
\varphi_t \circ M_0 : \bd W_0 \to \bd W_t$. Finally, let $V_t$ be the Zhukovski\u{i} preimage of
$W_t$, as defined in \S \ref{sec:zdc}, and $\psi_t: \bd V_0 \to \bd V_t$ be the Zhukovski\u{i} lift of $\zeta_t$ given by \lemref{zlift}. Observe that by \eqref{r=r},
$$
\rad(V_t,\infty)=\rad(W_t,\infty)=\rad(U_t,\infty)/|M'_t(\infty)|,
$$
so
$$
\log \rad(V_t,\infty) = \log \rad(U_t,\infty) + \log |s_t-q_t| - \log 4.
$$
By the equivalence (v) $\Longleftrightarrow$ (vi) established above, $t \mapsto \log
\rad(U_t,\infty)$ is harmonic in $\DD$. The same is true of $t \mapsto \log \rad(V_t,\infty)$ since $t \mapsto s_t-q_t$ is non-vanishing and holomorphic in $\DD$. Invoking (v) $\Longleftrightarrow$ (vi) once more, we see that the lifted motion $\psi_t$ respects the harmonic measure $\nu_t$ on $\bd V_t$ as seen from $\infty$: $(\psi_t)_* \nu_0 = \nu_t$. \vs

Now let $h_t=Z^{-1} \circ M_t^{-1} \circ g_t : (\Delta, \infty) \to (V_t,\infty)$ be the
corresponding Riemann map of $V_t$, where $Z^{-1}$ is the inverse of the biholomorphic restriction $Z:V_t \to W_t$. Consider the measurable sets
$$ 
X_t = h_t([1,a_t]) \ \subset \bd V_t \qquad (t \in \DD) 
$$
which satisfy $X_t=\psi_t(X_0)$. Decompose $X_t$ up to a set of harmonic
measure zero into the disjoint union of the measurable sets
\begin{align*}
X_t^1 & = X_t \cap \bd V_t^1 \\
X_t^2 & = X_t \cap \bd V_t^2.
\end{align*}
If $p \in X_t^2 \sm \{ -1, 1 \}$, \corref{joojoo} shows that the rays landing at
$p$ do not separate $-1=h_t(a_t)$ from $1=h_t(1)$, which means both angles in $h_t^{-1}(p)$ belong to $[1,a_t]$. It follows that there is a corresponding measurable decomposition $[1,a_t] = I_t^1 \cup I_t^2$ up to a set of Lebesgue measure zero, where
\begin{align*}
I_t^1 & = h_t^{-1}(X_t^1)  \\
I_t^2 & = h_t^{-1}(X_t^2). 
\end{align*}
Note that by \lemref{pe=e}, $\psi_t(X_0^n)=X_t^n$ for $n=1,2$. The relation $(\psi_t)_{\ast} \, \nu_0 = \nu_t$ now shows that
\begin{align*}
m(I_t^1) & = \nu_t(X_t^1) = \nu_0(X_0^1) = m(I_0^1) \\
m(I_t^2) & = \nu_t(X_t^2) = \nu_0(X_0^2)= m(I_0^2).
\end{align*}
Adding these equalities gives $m([1,a_t])=m([1,a_0])$, which shows $a_t=a_0$ for all
$t$. It follows that for every $t \in \DD$, $\varphi_t^{\circ}=\text{id}$ on the full-measure set $A_0^1 \cup A_0^2$, hence on $\TT$ by continuity. \vs
\end{proof}

\begin{theorem}
The conformal fitness condition (iii) in the Main Theorem can be replaced with the following weaker condition: \vs
\begin{enumerate}
\item[(iii$'$)]
There is a continuous family $g_t:(\DD,0) \to (U_t,c_t)$ of Riemann maps and a set $X \subset \TT$ of positive Lebesgue measure such that $g_t= \varphi_t \circ g_0$ on $X$ for all $t \in \DD$.  
\end{enumerate}
\end{theorem}

\begin{proof}
It suffices to show that (iii$'$) implies (i). As in the proof of (iii) $\Longrightarrow$ (i), after various normalizations we may assume that the family $\{ g_t \}_{t \in \DD}$ is uniformly bounded. Let $\gamma$ be any smooth closed curve in $\DD$ and define
$$
G(z)=\int_{\gamma} g_t(z) \, dt \qquad (z \in \DD).
$$
Then $G$ is holomorphic in $\DD$. Moreover, Lebesgue's dominated convergence theorem shows that for every $a \in X$,
\begin{align*}
G(a) & = \lim_{r \to 1} G(ra) = \lim_{r \to 1} \int_{\gamma} g_t(ra) \, dt \\
& = \int_{\gamma} \lim_{r \to 1} g_t(ra) \, dt = \int_{\gamma} g_t(a) \, dt \\
& = \int_{\gamma} \varphi_t(g_0(a)) \, dt =0,
\end{align*}
where the last equality follows from Cauchy's theorem since $t \mapsto
\varphi_t(g_0(a))$ is holomorphic. Thus, the bounded holomorphic function $G$ has
zero radial limit on the positive-measure set $X$. The theorem of F. and M. Riesz then shows that $G$ must be identically zero in $\DD$. Since $\gamma$ was arbitrary, it follows from Morera's theorem that for each fixed $z$, the map $t \mapsto g_t(z)$ is
holomorphic in $\DD$. 
\end{proof}

\bibliographystyle{amsplain}

\end{document}